\documentclass[a4paper,12pt]{article}
\usepackage{amsmath,amssymb,enumerate,amsthm,bbm}
\newcommand{\linktojournal}[1]{\relax}

\usepackage[whole]{bxcjkjatype}

\usepackage{hyperref}
\hypersetup{
setpagesize=false,
 bookmarksnumbered=true,%
 bookmarksopen=true,%
 colorlinks=true,%
 linkcolor=blue,
 citecolor=red
}

\newcommand{\R}{\mathbb{R}}

\newcommand{\Z}{\mathbb{Z}}
\newcommand{\N}{\mathbb{N}}

\newcommand{\pa}{\partial}

\DeclareMathOperator{\supp}{supp}
\newcommand{\lr}[1]{{}\langle{}#1{}\rangle{}}
\newcommand{\bigLR}[1]{{}\big\langle{}#1{}\big\rangle{}}
\newcommand{\BigLR}[1]{{}\Big\langle{}#1{}\Big\rangle{}}
\newcommand{\LR}[1]{{}\left\langle{}#1{}\right\rangle{}}
\newcommand{\rev}[1]{{{}#1{}}}

\newtheorem{theorem}{Theorem}[section]
\newtheorem{lemma}[theorem]{Lemma}
\newtheorem{proposition}[theorem]{Proposition}
\newtheorem{corollary}[theorem]{Corollary}
\theoremstyle{remark}
\newtheorem{remark}{Remark}[section]
\theoremstyle{definition}

\newtheorem{definition}{Definition}[section]

\numberwithin{equation}{section}

\makeatletter
\def\@cite#1#2{[{{\bfseries #1}\if@tempswa , #2\fi}]}
\makeatother                  

\begin{document}
\begin{center}
\Large{{\bf 
Asymptotic expansion for 
a class of
second-order evolution equations 
in the energy space 
and its applications
}}
\end{center}

\vspace{5pt}

\begin{center}
Motohiro Sobajima%
\footnote{
Department of Mathematics, 
Faculty of Science and Technology, Tokyo University of Science,  
2641 Yamazaki, Noda-shi, Chiba, 278-8510, Japan,  
E-mail:\ {\tt msobajima1984@gmail.com}}
\end{center}

\newenvironment{summary}{\vspace{.5\baselineskip}\begin{list}{}{%
     \setlength{\baselineskip}{0.85\baselineskip}
     \setlength{\topsep}{0pt}
     \setlength{\leftmargin}{12mm}
     \setlength{\rightmargin}{12mm}
     \setlength{\listparindent}{0mm}
     \setlength{\itemindent}{\listparindent}
     \setlength{\parsep}{0pt}
     \item\relax}}{\end{list}\vspace{.5\baselineskip}}
\begin{summary}
{\footnotesize {\bf Abstract.}
In this paper, we mainly discuss asymptotic profiles of 
solutions to a class of abstract second-order evolution equations 
of the form $u''+Au+u'=0$ in real Hilbert spaces, 
where $A$ is a nonnegative selfadjoint operator. 
The main result is the asymptotic expansion 
for all initial data belonging to the energy space, which is naturally expected.
This is an improvement for the previous work \cite{Sobajima_MathAnn}
(required a sufficient regularity).
As an application, we focus our attention to 
damped wave equations in an exterior domain $\Omega$ in $\R^N$ $(N\geq 2)$ with the Dirichlet boundary condition on $\pa \Omega$. 
By using the asymptotic expansion in the present paper,
we could derive the optimal decay rates of energy functional of solutions to damped wave equations. 
Moreover, some decay estimates of the local energy (energy functional restricted in a compact subset of $\Omega$) 
can be observed via the asymptotic expansion.
}
\end{summary}



\tableofcontents
\section{Introduction}

In this paper, we consider 
the initial-value problem of the following 
second-order evolution equation 
in a real Hilbert space $H$: 
\begin{equation}\label{intro:abst-dw}
\begin{cases}
u''(t)+Au(t)+u'(t)=0 & \text{in}\ \R_+, 
\\
(u,u')(0)=(u_0,u_1),
\end{cases}
\end{equation}
where $\R_+=(0,\infty)$ and $A$ is a nonnegative selfadjoint operator in $H$ 
(with the inner product $\lr{\cdot,\cdot}$ and the norm $\|\cdot\|$) endowed with domain $D(A)$. 
The pair $(u_0,u_1)$ belongs to 
the energy space $\mathcal{H}=D(A^{1/2})\times H$
and in this case
a weak solution 
\rev{%
$u\in C([0,\infty);D(A^{1/2}))\cap C^1([0,\infty);H)$ 
}
of the problem \eqref{intro:abst-dw} uniquely exists 
(see, e.g., Cazenave--Haraux \cite{CHbook} or Pazy \cite{Pazybook}).
The aim of the present paper is 
to provide the asymptotic expansion of weak solutions to \eqref{intro:abst-dw} 
for arbitrary initial data in the energy space $\mathcal{H}$. 

First, the energy functional 
\[
E(u;t)=\|u'(t)\|^2+\|A^{1/2}u(t)\|^2
\]
for the solution $u$ of \eqref{intro:abst-dw}
plays a central role to analyse properties of $u$. 
In particular, the energy identity 
\[
E(u;t)+2\int_0^t\|u'(s)\|^2\,ds=E(u;0), \quad t>0
\]
is well-known. 
In Ikehata--Nishihara \cite{IkNi2003}, 
they tried to analyse the problem \eqref{intro:abst-dw} 
via the energy method with a technique 
inspired by Morawetz \cite{Morawetz1961} 
and found that the weak solution $u$ of \eqref{intro:abst-dw}
has 
the asymptotic behavior $v(t)=e^{-tA}(u_0+u_1)$ with the estimate
\[
\|u(t)-v(t)e^{-tA}\|\leq C_{\delta}
\|(u_0,u_1)\|_{\mathcal{H}}(1+t)^{-1}\big(1+\log (1+t)\big)^{1/2+\delta}, 
\quad t\geq0
\]
for $(u_0,u_1)$ belonging to 
the energy space $\mathcal{H}$ 
(with the norm $\|(u_0,u_1)\|_{\mathcal{H}}^2=\|u_0\|_{D(A^{1/2})}^2+\|u_1\|^2$), where $\{e^{-tA}\}_{t\geq 0}$ is the $C_0$-semigroup on $H$ 
generated by $-A$. 
Actually, the discussion in \cite{IkNi2003} was motivated 
by the observation of the behavior of weak solutions to 
the initial-boundary value problem of the damped wave equation
\begin{equation}\label{intro:ext-dw}
\begin{cases}
\pa_t^2u(x,t)-\Delta u(x,t)+\pa_t u(x,t) =0 
&\text{in}\ \Omega \times \R_+, 
\\
u(x,t) =0 
&\text{on}\ \pa\Omega \times \R_+, 
\\
(u,\pa_tu)(x,0)=(u_0(x),u_1(x))
&\text{in}\ \Omega, 
\end{cases}
\end{equation}
where $\Omega=\R^N$ $(N\in \N)$ 
or $\Omega$ is a connected exterior domain in $\R^N$ ($N\geq 2$)
with a smooth boundary $\pa\Omega$; note that $\N$ denotes 
the set of all positive integers.
The logarithmic factor on the right-hand side of the above inequality 
is removed by Chill--Haraux \cite{ChHa2003}. 
In Radu--Todorova--Yordanov \cite{RTY2011}, 
applying spectral analysis, the authors 
proved the following estimate 
\[
\|u(t)-v(t)\|
\leq C
t^{-1}\big(\|e^{-\frac{t}{2}A}u_0\|+\|e^{-\frac{t}{2}A}u_1\|\big)
+e^{-\frac{t}{16}}\big(\|u_0\|+\|(1+A^{1/2})^{-1}u_1\|\big),\quad t\geq 0.
\]
As a consequence of these previous works, 
it has been clarified that the $C_0$-semigroup $\{e^{-tA}\}_{t\geq0}$ 
indicates the asymptotic behavior of weak solutions to \eqref{intro:abst-dw}. 

In the recent research Sobajima \cite{Sobajima_MathAnn}, 
it is shown that 
if the pair satisfies $(u_0,u_1)\in D(A^{n-\frac{1}{2}})$
for some $n\in\N$, then 
the solution $u$ of \eqref{intro:abst-dw} satisfies 
\begin{equation}\label{intro:expansion-reg}
\left\|u(t)-\sum_{\ell=0}^{n-1}v_\ell(t)\right\|
\leq 
C_nt^{-(n-\frac{1}{2})}\|(u_0,u_1)\|_{D(A^{n-\frac{1}{2}})\times D(A^{n-\frac{1}{2}})},
\quad t\geq 0,
\end{equation}
where $v_0(t)=e^{-tA}(u_0+u_1)$ and for $\ell\in\N$, 
\begin{equation}\label{intro:expansion-profiles}
v_\ell(t)
=
A^\ell\left(
\sum_{j=0}^\ell
\begin{pmatrix}
2\ell\\\ell+j
\end{pmatrix}\frac{(-tA)^j}{j!}e^{-tA}(u_0+u_1)
-
\sum_{k=0}^{\ell-1}
\begin{pmatrix}
2\ell-1\\ \ell+k
\end{pmatrix}\frac{(-tA)^k}{k!}e^{-tA}u_0
\right).
\end{equation}
The above work can be regarded as 
a generalization of the result in Takeda \cite{Takeda2015} 
about the asymptotic expansion (with a different representation) 
for the Cauchy problem of the damped wave equation
\begin{equation}\label{intro:dw}
\begin{cases}
\pa_t^2u(x,t)-\Delta u(x,t)+\pa_t u(x,t) =0 
&\text{in}\ \R^N \times \R_+, 
\\
(u,\pa_tu)(x,0)=(u_0(x),u_1(x))
&\text{in}\ \R^N. 
\end{cases}
\end{equation}
It is remarkable that 
the asymptotic expansion in \cite{Sobajima_MathAnn} 
is also applicable to the exterior problem of the damped wave equation \eqref{intro:ext-dw}.

Of course, there are many previous works 
dealing with the problem \eqref{intro:dw} 
via the solution formula with the Fourier transform.
In the pioneering work \cite{Matsumura1976} by Matsumura, 
it is shown that 
the solution $u$ of \eqref{intro:dw} satisfies 
\[
\|\pa_t^k\pa_x^\alpha u (t)\|_{L^2}
\leq 
C(1+t)^{-\frac{N}{2}(\frac{1}{q}-\frac{1}{2})-k-\frac{|\alpha|}{2}}
\|(u_0,u_1)\|_{(H^{m}\cap L^q)\times (H^{m-1}\cap L^q)}, 
\quad 
t\geq 0
\]
(and also $L^\infty$-estimates) with $q\in [1,2]$ and $m=k+|\alpha|$, 
where $k\in \Z_{\geq 0}$, $\alpha\in \Z_{\geq 0}^N$ 
($\Z_{\geq 0}=\N\cup\{0\}$ and $|\alpha|=\sum_{k=1}^N\alpha_k$).
Here we have used the notations of the Lebesgue spaces $L^q=L^q(\R^N)$ 
and the Sobolev spaces $H^m=H^m(\R^N)$. 
Although there are many subsequent papers 
for asymptotic behavior of weak solutions to \eqref{intro:dw} and also \eqref{intro:ext-dw}, 
we will not enter details 
(For subsequent papers related to 
the Matsumura estimates and the asymptotic behavior
in various situations and their applications 
to nonlinear problems, see e.g., Hsiao--Liu \cite{HsLi1992}, 
Kawashima--Nakao--Ono \cite{KaNaOn1995},
Karch \cite{Karch2000},
Yang--Milani \cite{YaMi2000}, 
Ikehata--Ohta \cite{IO2002},
Nishihara \cite{Nishihara2003}, 
Narazaki \cite{Narazaki2004},
Hosono--Ogawa \cite{HoOg2004},
Ikeda--Inui--Okamoto--Wakasugi \cite{IIOW2019},
Ikeda--Taniguchi--Wakasugi \cite{ITWpre} 
and there references therein). 

It should be 
\rev{%
pointed out
}
that 
the asymptotic behavior
can be found in \cite{IkNi2003},\cite{ChHa2003} and also \cite{RTY2011} 
with initial data in the energy space $\mathcal{H}$, 
but the estimate \eqref{intro:expansion-profiles} with $n=1$
requires the additional regularity $u_1\in D(A^{1/2})$.
The purpose of the present paper is to
fill the gap explained above, 
that is, 
to justify the asymptotic expansion 
with initial data in energy space $\mathcal{H}$ {\it without any additional regularity}. 

To state the result of the present paper, 
we give the precise definition of weak solutions 
to \eqref{intro:abst-dw} and the corresponding asymptotic profiles.
\begin{definition}[Definition of weak solutions]
Let $A$ be nonnegative and selfadjoint in the Hilbert space $H$ with domain $D(A)$. 
For $(u_0,u_1)\in \mathcal{H}=D(A^{1/2})\times H$, 
the function $u:[0,\infty)\to H$ is called a {\it weak solution} of \eqref{intro:abst-dw} if 
$u$ belongs to the class $C^1([0,\infty);H)\cap C([0,\infty);D(A^{1/2}))$ and 
$\mathcal{U}(\cdot)=(u(\cdot),u'(\cdot))\in C([0,\infty);\mathcal{H})$ satisfies
$\mathcal{U}(t)=e^{t\mathcal{L}}(u_0,u_1)$ for all $t\geq 0$, 
where $\{e^{t\mathcal{L}}\}_{t\geq0}$ 
is the $C_0$-semigroup on $\mathcal{H}$ generated by 
$\mathcal{L}(f,g)=(g,-Af-g)$ endowed with domain 
$D(\mathcal{L})=D(A)\times D(A^{1/2})$. 
\end{definition}
\begin{definition}[{\cite{Sobajima_MathAnn}}]\label{def:asymptotics}
For $(u_0,u_1)\in \mathcal{H}$, 
define a family of functions $\{V_\ell (\cdot)\}_{\ell\in \Z_{\geq 0}}$ as
\begin{align*}
v_0(t)
&=e^{-tA}(u_0+u_1), 
\\
v_\ell(t)
&=
A^\ell\left(
\sum_{j=0}^\ell
\begin{pmatrix}
2\ell\\\ell+j
\end{pmatrix}\frac{(-tA)^j}{j!}e^{-tA}(u_0+u_1)
-
\sum_{k=0}^{\ell-1}
\begin{pmatrix}
2\ell-1\\ \ell+k
\end{pmatrix}\frac{(-tA)^k}{k!}e^{-tA}u_0
\right), \quad \ell\in \N.
\end{align*}
Also we define another family $\{V_n(\cdot)\}_{n\in\Z_{\geq0}}$ as 
$V_0(t)=0$ and $V_{n}(t)=\displaystyle\sum_{\ell=0}^{n-1}v_{\ell}(t)$ for $n\in\N$.
\end{definition}

The following is the main result of the present paper 
describing asymptotic expansion for general initial data.  

\begin{theorem}\label{intro:thm:abst-expansion}
For every $n\in \Z_{\geq 0}$, 
there exists a positive constant $C_{1,n}$ 
(depending only on $n$) 
such that the following assertion holds:
For every $(u_0,u_1)\in \mathcal{H}$, 
the corresponding weak solution $u$ of \eqref{intro:abst-dw}
satisfies 
\begin{align*}
\left\|
u(t)-V_{n}(t)
\right\|^2
+
tE\left(u-V_n;t\right)
&\leq C_{1,n}t^{-2n}\|(u_0,u_1)\|_{\mathcal{H}}^2,\quad t\geq 1
\end{align*}
where $\{V_n(\cdot)\}_{n\in \Z_{\geq0}}$ is given in Definition \ref{def:asymptotics}.
\end{theorem}
\begin{remark}
For the case $m=1$, 
the estimate in Theorem \ref{intro:thm:abst-expansion} 
matches the result of \cite{ChHa2003} and \cite{RTY2011}.
We can also see the decay estimate for 
the energy functional $E(u;\cdot)$
for solution $u$ to \eqref{intro:abst-dw}. 
\end{remark}
\begin{remark}
The short-time estimate $(t\ll 1)$ 
depends on the regularity of the initial data.
In contrast, for the large-time behavior,  
we can observe all asymptotic profiles 
without any additional regularity.
\end{remark}
Incidentally, the above asymptotic expansion 
stated in Theorem \ref{intro:thm:abst-expansion} has 
several applications to the exterior problem 
of the damped wave equation \eqref{intro:ext-dw}.
In this case, 
the corresponding choice is $H=L^2(\Omega)$ and $A=-\Delta_{\Omega}$, where $\Delta_{\Omega}$ is 
so-called Dirichlet Laplacian in $L^2(\Omega)$
endowed with domain $D(\Delta_{\Omega})=H^2(\Omega)\cap H_0^1(\Omega)$. Hereafter, we will use the corresponding energy space 
\[
\mathbf{E}(\Omega)=D((-\Delta_{\Omega})^{1/2})\times H=H_0^1(\Omega)\times L^2(\Omega)
\]
and 
\[
\mathbf{L}^q(\Omega)=L^q(\Omega)\times L^q(\Omega)
\] 
for additional integrability conditions.
As is expected, the Dirichlet heat semigroup $\{e^{t\Delta_{\Omega}}\}_{t\geq 0}$
plays a crucial role for the analysis of \eqref{intro:ext-dw}. 
To describe the detail, we introduce 
the 
\rev{%
energy functional
}
$E_{\Omega}(u;\cdot)$ for the problem \eqref{intro:ext-dw} as 
\[
E_{\Omega}(u;t)=\int_{\Omega}\Big(|\pa_tu(x,t)|^2+|\nabla u(x,t)|^2\Big)\,dx, \quad t\geq 0
\]
for $u\in C^1([0,\infty);L^2(\Omega))\cap C([0,\infty);H_0^1(\Omega))$. 
Applying Theorem \ref{intro:thm:abst-expansion} with the usual $L^p$-$L^q$ estimate 
$(1\leq p\leq q\leq \infty)$
\[
\|e^{t\Delta_\Omega}f\|_{L^q(\Omega)}
\leq Ct^{-\frac{N}{2}(\frac{1}{p}-\frac{1}{q})}\|f\|_{L^p(\Omega)}, \quad t>0,
\]
we can find the following assertion. 

\begin{corollary}\label{intro:cor:decayest}
For every $q\in [1,2]$ and $m\in \Z_{\geq 0}$, 
there exists a positive constant $C_{2,q,m}$ such that 
the following assertion holds:
If $(u_0,u_1)\in \mathbf{E}(\Omega)\cap \mathbf{L}^q(\Omega)$, 
then the corresponding weak solution of \eqref{intro:ext-dw}
satisfies 
\begin{gather*}
\left\|u(t)-V_m(t)
\right\|_{L^2(\Omega)}^2+
tE_{\Omega}\left(u-V_m;t\right)
\leq C_{2,q,m}
t^{-N(\frac{1}{q}-\frac{1}{2})-2m}
\|(u_0,u_1)\|_{\mathbf{E}(\Omega)\cap \mathbf{L}^q(\Omega)}^2,
\quad t\geq 1, 
\end{gather*}
where $\{V_m\}_{m\in\Z_{\geq0}}$ is given in Definition \ref{def:asymptotics} 
with $A=-\Delta_{\Omega}$ in $L^2(\Omega)$. 
\end{corollary}
\begin{remark}
Ono \cite{Ono2003} gives the estimates 
\[
\|u(t)\|_{L^2(\Omega)}^2+
(1+t)
E_{\Omega}(u;t)
\leq 
\begin{cases}
C(1+t)^{-\frac{N}{2}}
&\text{if}\ N\geq 3, 
\\
C_\delta(1+t)^{-1+\delta}
&\text{if}\ N=2
\end{cases}
\]
(note that $C_\delta\to \infty$ as $\delta\to 0$) 
under the assumption $(u_0,u_1) \in \mathbf{E}(\Omega)\cap \mathbf{L}^1(\Omega)$, 
which could be understood as the one of generalization for the Matsumura estimates.  
For the two-dimensional exterior problem, 
Saeki--Ikehata \cite{SaekiIkehata2000} found 
the decay estimate $E_{\Omega}(u;t)\leq C(1+t)^{-2}$
also for non-compactly supported initial data
in $\mathbf{E}(\Omega)$ with 
the additional condition $d_\Omega(u_0+u_1)\in L^2(\Omega)$
with $d_\Omega(x)=|x|\log(\frac{2|x|}{r_\Omega})$ and $r_\Omega=\inf\{|x|\;;\;x\in \pa\Omega\}$.
\end{remark}

The case of two-dimensional exterior domains is something exceptional.
Actually, in Grigor'yan and Saloff-Coste \cite{GS2002}, 
it is observed that 
Dirichlet heat semigroups $\{e^{t\Delta_\Omega}\}_{t\geq 0}$ in two-dimensional exterior domains
satisfy a peculiar behavior which originates from 
the recurrence property of the Brownian motion in $\R^2$.  This can be translated into 
the following inequality for $\{e^{t\Delta_{\Omega}}\}_{\geq 0}$:
\begin{equation}\label{intro:GS2002-ineq}
\|e^{t\Delta_\Omega}f\|_{L^2(\Omega)}
\leq Ct^{-\frac{1}{2}}\big(1+\log(1+t)\big)^{-1}\left\|
\left(1+\log\frac{|x|}{r_\Omega}\right) f\right\|_{L^1(\Omega)}, \quad t>0
\end{equation}
(for its alternative proof via the comparison principle, 
see also Appendix in Ikeda--Sobajima--Taniguchi--Wakasugi \cite{ISTWpre}).
If we use \eqref{intro:GS2002-ineq} instead of the usual $L^p$-$L^q$ estimates, 
then the following assertion can be found out.

\begin{corollary}\label{intro:cor:decayest-2d}
Let $\Omega$ be a two-dimensional exterior domain. 
Then for every $m\in \Z_{\geq 0}$, 
there exists a positive constant $C_{3,m}$ such that 
the following assertion holds:
If $(u_0,u_1)\in \mathbf{E}(\Omega)$ with 
\[
M=\left\|\Big(1+\log\frac{|x|}{r_\Omega}\Big)u_0
\right\|_{L^1(\Omega)}+
\left\|\Big(1+\log\frac{|x|}{r_\Omega}\Big)u_1\right\|_{L^1(\Omega)}<\infty,
\] 
then the corresponding weak solution $u$ of \eqref{intro:ext-dw}
satisfies 
\begin{gather*}
\left\|u(t)-V_m(t)
\right\|_{L^2(\Omega)}^2+
tE_{\Omega}\left(u-V_m;t\right)
\leq C_{3,m}
t^{-1-2m}(\log t)^{-2}
\big(\|(u_0,u_1)\|_{\mathbf{E}(\Omega)}^2+M^2\big),\quad t\geq 2
\end{gather*}
where $\{V_m(\cdot)\}_{m\in\Z_{\geq0}}$ is given in Definition \ref{def:asymptotics} 
with $A=-\Delta_{\Omega}$ in $L^2(\Omega)$. 
\end{corollary}
\begin{remark}
The case $m=0$ can be found in Ikeda--Sobajima--Taniguchi--Wakasugi \cite{ISTWpre}.
Corollary \ref{intro:cor:decayest-2d} also provides 
a slight improvement in the estimates for $m\geq 1$.   
\end{remark}

We also provides the optimality of these estimates in the following sense. 
The essential tool for the proof 
is the test function method with positive harmonic functions satisfying 
Dirichlet boundary condition 
(similar treatment can be found in Sobajima \cite{Sobajima_Nash} for 
one-dimensional Schr\"odinger semigroups). 

\begin{proposition}\label{intro:prop:sharpness1}
Let $u$ be the weak solution of \eqref{intro:ext-dw} 
with $(u_0,u_1)\in C_0^\infty(\Omega)\times C_0^\infty(\Omega)$
satisfying $u_0\geq 0, u_1\geq 0$ and $u_0+u_1\not\equiv 0$.
Then the following assertions hold: 
\begin{itemize}
\item[\bf (i)]
If $\Omega=\R^N$ $(N\in\N)$ or 
$\Omega$ is an exterior domain in $\R^N$ $(N\geq 3)$, 
then 
the following quantities are both finite and 
\rev{%
strictly
}
positive:
\[
\limsup_{t\to\infty} \Big(t^{\frac{N}{2}}\|u(t)\|_{L^2(\Omega)}^2\Big), 
\quad 
\limsup_{t\to\infty} \Big(t^{\frac{N}{2}+1}E_{\Omega}(u;t)\Big).
\]
\item[\bf (ii)]If 
$\Omega$ is a two-dimensional exterior domain, 
then the following quantities are both finite and 
\rev{%
strictly
}
positive:
\[
\limsup_{t\to\infty} \Big(t(\log t)^2\|u(t)\|_{L^2(\Omega)}^2\Big), 
\quad 
\limsup_{t\to\infty} \Big(t^2(\log t)^2E_{\Omega}(u;t)\Big).
\]
\end{itemize}
\end{proposition}

\begin{remark}
We will prove the same conclusion with slightly generalized initial data in Section \ref{sec:damped}.
\end{remark}
Another application of Theorem \ref{intro:thm:abst-expansion} 
is the estimation for 
a localized version of 
the energy functional 
\[
E_{\Omega,R}(u;t)=\int_{\Omega\cap B(0,R)}\Big(|\pa_tu(x,t)|^2+|\nabla u(x,t)|^2\Big)\,dx, \quad t\geq 0
\]
with $R>R_{\Omega}=\sup\{|x|\;;\;x\in \pa \Omega\}$ (or $R_{\R^N}=0$), 
where $B(y,r)$ denotes the $N$-dimensional ball 
centred at $y\in \R^N$ with radius $r$. This is so-called ``local energy''. 
The consequence is as follows:

\begin{proposition}\label{intro:prop;local}
Let $\Omega$ be an exterior domain $(N\geq 2)$. 
Let $u$ be the weak solution of \eqref{intro:ext-dw}
with $(u_0,u_1)\in C_0^\infty(\Omega)\times C_0^\infty(\Omega)$.
Then there exists a positive constant $C_{4}$
such that 
$u$ satisfies for every $R>R_\Omega$, 
\[
E_{\Omega,R}(u;t)\leq 
\begin{cases}
C_{4} R^{N-2}
t^{-N}
& \text{if}\ N\geq 3, 
\\
C_4 \left(1+\log \frac{R}{r_\Omega}\right)^2
t^{-2}(\log t)^{-4}
& \text{if}\ N=2
\end{cases}
\]
for every $t\geq 2$, where $R_\Omega=\sup\{|x|\;;\;x\in \pa \Omega\}$
and $r_{\Omega}=\inf\{|x|\;;\;x\in \pa \Omega\}$.
\end{proposition}
\begin{remark}
In Shibata \cite{Shibata1983} (for $N\geq 3$) 
and Dan--Shibata \cite{DaSh1995} (for $N\geq 2$),  
the same estimate has been derived without logarithmic 
\rev{%
factor.
}
More precisely, the following estimates were shown:
If $(u_0,u_1)\in \mathbf{E}(\Omega)$ satisfies $\supp u_0 \cup \supp u_1\subset B(0,R)$, then 
for every $t\geq 1$, 
\[
E_{\Omega,R}(u;t)\leq C_Rt^{-N}\|(u_0,u_1)\|_{\mathbf{E}(\Omega)}^2.
\]
The method in \cite{DaSh1995} (and also \cite{Shibata1983}) 
is the asymptotic analysis of resolvent operators 
related to the damped wave equation. 
Proposition \ref{intro:prop;local} 
for the case $N\geq 3$ 
may be regarded as 
an alternative proof of their result. 
A novelty here is the appearance of the logarithmic factor $(\log t)^{-4}$ in the estimate for the two-dimensional case, 
which is slightly 
faster than the decay rate of the energy functional $E_{\Omega}(u;\cdot)$ in  Corollary \ref{intro:cor:decayest-2d}. 
But we do not know whether its decay rate is optimal or not. 
\end{remark}

Here we briefly explain the crucial idea in the present paper. 
The derivation of asymptotic profiles has been provided in \cite{Sobajima_MathAnn}.
However, if we directly apply the decomposition in \cite{Sobajima_MathAnn}, 
then the initial data are required to be regular enough.
Therefore we further introduce a procedure for the regularization. 
Here we use a pair of reasonable operators $I_n,J_n:H\to D(A^n)$ satisfying $1_H=I_n+A^{n}J_n$, 
where $1_H$ denotes the identity map from $H$ to itself (see Lemma \ref{lem:unit-decomposition}). 	 
This gives the decomposition $u=u_*+A^n u_{**}$ with the respective initial data 
\[
(u_*,u_*')(0)=(I_nu_0,I_nu_1), 
\quad 
(u_{**},u_{**}')(0)=(J_nu_0,J_nu_1).
\]
This expression is reasonable because $u_*$ and $u_{**}$ are regular enough ($u_*, u_{**}\in D(A^{n+\frac{1}{2}})$)
and $A^{n}u_{**}$ can be expected to have an extra decay property comes from the factor $A^n$.
At this moment, we can apply the procedure in \cite{Sobajima_MathAnn} to 
$u_*$. 
The rest is a straightforward computation based on the energy method.

For the optimality of decay rates for energy functionals, 
we use the test function method 
with 
positive harmonic functions 
satisfying the Dirichlet boundary condition
for {\it linear} equations 
together with the Gagliardo--Nirenberg and Nash inequalities. 
In the case of two-dimensional exterior domain, 
we further need a Nash-type inequality involving a logarithmic weight function 
found in \cite{ISTWpre}.

The present paper is organized as follows. 
At the beginning of Section \ref{sec:abst}, 
we start with the basic facts 
in the analysis of the second-order evolution equation \eqref{intro:abst-dw} 
from the view point of the auxiliary quantity
\begin{equation}\label{intro:aux}
\|u(t)\|_{\mathcal{E}^\sharp}^2
=
\|A^{1/2}u(t)\|^2+\left\|u'(t)+\frac{1}{2}u(t)\right\|^2+\frac{1}{4}\|u(t)\|^2
\end{equation}
(a similar explanation can be found in Sobajima \cite{Sobajima_Effect}). 
In the middle of Section \ref{sec:abst}, 
we would give an alternative proof of 
asymptotic expansion for regular initial data stated in \cite{Sobajima_MathAnn}
via the use of the auxiliary quantity $\|u(t)\|_{\mathcal{E}^\sharp}^2$. 
Then at the end of Section \ref{sec:abst}, 
we shall give a proof of Theorem \ref{intro:thm:abst-expansion} 
via the regularization with $I_n$ and $J_n$ explained above.

As applications of Theorem \ref{intro:thm:abst-expansion}, 
we discuss the profile of weak solutions to the exterior problem of 
the damped wave equation \eqref{intro:ext-dw} in Section \ref{sec:damped}. 
First we prove the decay estimates for $L^2$-norm  and the energy $E_{\Omega}(u;\cdot)$. 
We also show the optimality of their decay rates 
via the test function method with positive harmonic functions satisfying the Dirichlet boundary condition. 
Finally, we discuss decay estimates of the local energy 
$E_{\Omega,R}(u;t)$ for weak solutions to \eqref{intro:ext-dw}.

\section{Analysis for the abstract evolution equation}\label{sec:abst}

Here we consider the initial-value problem of the second-order 
evolution equation 
in a real (or even complex) Hilbert space $H$ 
(with the inner product $\lr{\cdot,\cdot}$
and the norm $\|\cdot\|$), 
which is of the form
\begin{align}\label{eq:abst-dw}
\begin{cases}
u''(t)+Au(t)+u'(t)=0, \quad t\in\R_+, 
\\
(u,u')(0)=(u_0,u_1)\in D(A^{1/2})\times H.
\end{cases}
\end{align}
Here $A$ is a nonnegative selfadjoint operator in $H$ 
endowed with domain $D(A)$. 
We discuss the asymptotic expansion of weak solutions to \eqref{eq:abst-dw}. 
\subsection{Basic facts with a derivation}
As usual, existence and uniqueness of weak solutions to \eqref{eq:abst-dw}
is obtained by the Hille--Yosida theorem. 
Here we use the following space and operator; 
note that the choice of the norm for the product space is 
slightly different from the usual one for but equivalent to it.

\begin{definition}
\begin{itemize}
\item[\bf (i)] 
Define $\mathcal{H}=D(A^{1/2})\times H$ as a Hilbert space with the inner product 
\[
\bigLR{(f_1,g_1),(f_2,g_2)}_\mathcal{H}=
\lr{A^{1/2}f_1,A^{1/2}f_2}+
\frac{1}{4}\lr{f_1,f_2}+
\BigLR{g_1+\frac{1}{2}f_1,g_2+\frac{1}{2}f_2}
\]
for $(f_1,g_1),(f_2,g_2)\in\mathcal{H}$; note that the norm $\|\cdot\|_{\mathcal{H}}$ 
is equivalent to the usual norm:
\[
\|A^{1/2}f\|^2+\frac{1}{6}\|f\|^2+\frac{1}{4}\|g\|^2
\leq \|(f,g)\|_{\mathcal{H}}^2\leq 
\|A^{1/2}f\|^2+\|f\|^2+\frac{3}{2}\|g\|^2. 
\]
\item[\bf (ii)] 
Define the densely defined operator $\mathcal{L}:D(\mathcal{L})\,(\subset\mathcal{H})\to \mathcal{H}$ as
\begin{align*}
\begin{cases}
D(\mathcal{L})=D(A)\times D(A^{1/2})
\\
\mathcal{L}(f,g)=(g,-Af-g);
\end{cases}
\end{align*}
note that by using $\mathcal{L}$ we can reduce the second-order 
evolution equation \eqref{eq:abst-dw} 
to the first-order equation $\frac{d}{dt}(u,\widetilde{u})=\mathcal{L}(u,\widetilde{u})$ equipped with the initial condition $(u,\widetilde{u})|_{t=0}=(u_0,u_1)$. 
\item[\bf (iii)] 
We also define for $\xi\in C^1([0,\infty);H)\cap C([0,\infty);D(A^{1/2}))$, 
\begin{align*}
\|\xi(t)\|_{\mathcal{E}^\sharp}
&=
\|(\xi(t),\xi'(t))\|_{\mathcal{H}}.
\end{align*}
\end{itemize}
\end{definition}
\begin{remark}
The following identity suggests that 
the inner product $\lr{\cdot,\cdot}_{\mathcal{H}}$ 
might represents a reasonable dissipative structure for the problem \eqref{eq:abst-dw}:
for every $(f,g)\in D(\mathcal{L})$, 
\begin{align*}
\bigLR{\mathcal{L}(f,g),(f,g)}_\mathcal{H}
&=
\bigLR{(g,-Af-g),(f,g)}_\mathcal{H}
\\
&=
\lr{A^{1/2}g,A^{1/2}f}
+\BigLR{-Af-\frac{1}{2}g,g+\frac{1}{2}f}+\frac{1}{4}+\lr{g,f}
\\
&=-\frac{1}{2}\Big(\|A^{1/2}f\|^2+\|g\|^2\Big). 
\end{align*}
\end{remark}\label{rem:accretive}
The solvability of the abstract problem \eqref{eq:abst-dw} 
is verified by the Hille--Yosida theorem.
\begin{proposition}\label{prop:abst:solve}
The operator $-\mathcal{L}$ is $m$-accretive in $\mathcal{H}$,
and therefore,
$\mathcal{L}$ generates a $C_0$-semigroup $\{e^{-t\mathcal{A}}\}_{t\geq0}$
of contractions on $\mathcal{H}$. In particular,  for every $n\in\Z_{\geq 0}$, one has
\[
e^{t\mathcal{L}}(u_0,u_1)
\in 
\bigcap_{k=0}^n C^{n-k}([0,\infty);D(\mathcal{L}^k))
\]
provided if $(u_0,u_1)\in D(\mathcal{L}^n)$.
\end{proposition}

The following lemma is 
the merit of the use of 
the quantity $\|u(t)\|_{\mathcal{E}^\sharp}^2$.
\begin{lemma}\label{lem:identities}
Let $u$ be the weak solution of \eqref{eq:abst-dw} 
with $(u_0,u_1)\in D(\mathcal{L})$. 
Then one has 
\begin{gather}
\label{eq:en}
\frac{d}{dt}E(u;t)
+
2\|u'(t)\|^2=0, \quad t\geq 0,
\\
\label{eq:en2}
\frac{d}{dt}
\|u(t)\|_{\mathcal{E}^\sharp}^2
+
E(u;t)
=0, \quad t\geq0.
\end{gather}
\end{lemma}
\begin{proof}
For the convenience, we also give a proof of the well-known identity \eqref{eq:en}.
By the standard computation with the equation in \eqref{eq:abst-dw}, we have
\begin{align*}
\frac{d}{dt}E(u;t)
&=2\lr{u'(t),u''(t)}+2\lr{A^{1/2}u'(t),A^{1/2}u(t)}
\\
&=2\lr{u'(t),u''(t)+Au(t)}
\\
&=-2\|u'(t)\|^2.
\end{align*}
For the second identity \eqref{eq:en2}, we argue in the following way:
\begin{align*}
\frac{d}{dt}\|u(t)\|_{\mathcal{E}^{\sharp}}^2
=\frac{d}{dt}\|e^{t\mathcal{L}}(u_0,u_1)\|_{\mathcal{H}}^2
=
2\LR{\mathcal{L}e^{t\mathcal{L}}(u_0,u_1),e^{t\mathcal{L}}(u_0,u_1)}_{\mathcal{H}}
=
-E(u;t),
\end{align*}
where we have used the computation in Remark \ref{rem:accretive}.
\end{proof}

Therefore we can observe 
the well-known estimate of $\|u(t)\|^2+(1+t)E(u;t)$ 
for the solution $u$ to \eqref{eq:abst-dw} 
via the identities in Lemma \ref{lem:identities}. 
\begin{lemma}\label{lem:basic}
Let $u$ be the weak solution of \eqref{eq:abst-dw} with $(u_0,u_1)\in\mathcal{H}$. 
Then 
\[
2\|u(t)\|_{\mathcal{E}^\sharp}^2
+
tE(u;t)
+2	\int_0^t\Big(\|A^{1/2}u(s)\|^2+(1+s)\|u'(s)\|^2\Big)\,ds
\leq 
2\|(u_0,u_1)\|_\mathcal{H}^2, \quad t\geq 0.
\]
\end{lemma}
Decay estimates for higher derivatives of the weak solution $u$
to \eqref{eq:abst-dw} is proved via an iterative procedure. 
The second lemma is helpful 
to derive a desired estimate for each step in such an iteration.

\begin{lemma}\label{lem:for-itelation}
Let $w$ be the weak solution of $w''+Aw+w'=0$ 
with $(w,w')(0)=(w_0,w_1)\in \mathcal{H}$.
Then for every $k>0$, there exists 
a positive constant $C_{5,k}$ 
depending only on $k$ 
such that 
\begin{align*}
&(1+t)^{k}\Big(\|w(t)\|_{\mathcal{E}^\sharp}^2
+tE(w;t)
\Big)
+\int_0^t (1+s)^{k}\Big(\|A^{1/2}w(s)\|^2+(1+s)\|w'(s)\|^2\Big)\,ds
\\
&\leq
C_{5,k}\left(
\|(w_0,w_1)\|_{\mathcal{H}}^2+\int_0^t (1+s)^{k-1}\|w(s)\|^2\,ds
\right), \quad t\geq 0.
\end{align*}
\end{lemma}
\begin{proof}
This is a direct consequence of 
the following computation:
\begin{align*}
&\frac{d}{dt}\left[(4k+t)^{k}\|w(t)\|_{\mathcal{E}^\sharp}^2
+\frac{(4k+t)^{k+1}}{2(k+1)}E(w;t)
\right]
\\
&=
(4k+t)^{k}\left[
\frac{k}{4k+t}\|w(t)\|_{\mathcal{E}^\sharp}^2
-E(w;t)
+\frac{1}{2}E(w;t)
-\frac{4k+t}{k+1}\|w'(t)\|^2
\right]
\\
&\leq 
(4k+t)^{k}\left[
\frac{k}{4k+t}\|w\|^2
-\frac{1}{8}E(w;t)
-\frac{4k+t}{k+1}\|w'(t)\|^2
\right].
\end{align*}
Integrating it over $[0,t]$, we deduce the desired estimate.
\end{proof}
Decay estimates for $u^{(n)}=\frac{d^n}{dt^n}u$ and $A^{n/2}u$ 
of the weak solution $u$ to \eqref{eq:abst-dw} are stated as follows. 

\begin{lemma}\label{lem:regularity1}
For every $n\in\N$, there exists a positive constant 
$C_{6,n}$ such that the following holds:
for every $(u_0,u_1)\in D(\mathcal{L}^n)$, 
the corresponding 
\rev{%
(strong)
}
solution $u$ of \eqref{eq:abst-dw} 
satisfies
\begin{equation}\label{eq:lem:t-deri}
\Big(\|u^{(n)}(t)\|^2+(1+t)
E(u^{(n)};t)\Big)
\leq C_{6,n}(1+t)^{-2n}
\|(u_0,u_1)\|_{D(\mathcal{L}^n)}^2, \quad t\geq 0
\end{equation}
and 
\begin{equation}\label{eq:lem:A^{n/2}}
\Big(\|A^{n/2}u(t)\|^2+(1+t)
E(A^{n/2}u;t)\Big)
\leq C_{6,n}
(1+t)^{-n}
\|(u_0,u_1)\|_{D(\mathcal{L}^n)}^2, \quad t\geq 0.
\end{equation}
\end{lemma}
\begin{proof}
We first show \eqref{eq:lem:t-deri} 
by proving 
\begin{align}
\nonumber 
&(1+t)^{2n}\Big(
\|u^{(n)}(t)\|_{\mathcal{E}^\sharp}^2+
tE(u^{(n)};t)
\Big)
\\
&
\label{eq:itelation-deri}
+\int_0^t(1+s)^{2n}\Big(\|A^{1/2}u^{(n)}(s)\|^2+(1+s)\|u^{(n+1)}(s)\|^2\Big)\,ds
\leq 
\widetilde{C}_{n}\|(u_0,u_1)\|_{D(\mathcal{L}^n)}^2
\end{align}
for some series of positive constants $\{\widetilde{C}_{n}\}_{n\in\Z_{\geq 0}}$.
We argue by induction as explained above. 
In view of Lemma \ref{lem:basic}, we already have \eqref{eq:itelation-deri} with $n=0$. 
For the estimate \eqref{eq:itelation-deri} with $n\in\N$, 
we assume the validity of \eqref{eq:itelation-deri} with $n=m-1$ holds with $m\in\N$.  
Then Lemma \ref{lem:for-itelation} with $(w,k)=(u^{(m+1)},2m+2)$
gives for every $t\geq 0$, 
\begin{align*}
&(1+t)^{2m}\Big(\|u^{(m)}(t)\|_{\mathcal{E}^\sharp}^2
+tE(u^{(m)};t)
\Big)
\\
&+\int_0^t (1+s)^{2m}\Big(\|A^{1/2}u^{(m)}(s)\|^2+(1+s)\|u^{(m+1)}(s)\|^2\Big)\,ds
\\
&\leq
C_{5,2m}\left(
\|(u^{(m)}(0),u^{(m+1)}(0))\|_{\mathcal{H}}^2
+\int_0^t (1+s)^{2m-1}\|u^{(m)}(s)\|^2\,ds
\right)
\\
&\leq
C_{5,2m}
\left(
\|\mathcal{L}^m(u_0,u_1)\|_{\mathcal{H}}^2
+\widetilde{C}_{m-1}\|(u_0,u_1)\|_{D(\mathcal{L}^{m-1})}^2
\right).
\end{align*}
For \eqref{eq:lem:A^{n/2}}, 
a similar induction argument via Lemma \ref{lem:for-itelation} with 
$(w,k)=(A^{m/2}u,m)$ provides the desired inequality. 
\end{proof}

\subsection{Asymptotic expansion}
Here we consider the asymptotic expansion for weak solutions to \eqref{eq:abst-dw} as $t\to \infty$. 
Namely, we would explain how to construct a family $\{v_\ell(t)\}_{\ell\in \Z_{\geq0}}$ 
satisfying 
for every $n\in \Z_{\geq0}$, 
\begin{gather*}
\|v_n(t)\|=O(t^{-n}),\quad \left\|u(t)-V_n(t)\right\|= O(t^{-n}) 
\quad \text{as}\ t\to \infty,
\end{gather*}
where $V_0(t)=0$ and $V_n(t)=\sum_{\ell=0}^{n-1}v_{\ell}(t)$ for $n\in\N$.
\subsubsection{Asymptotic expansion for regular initial data (revisited)}\label{subsec:pre}

If the initial data are regular enough, then 
the asymptotic profiles of 
(regular) solutions to \eqref{eq:abst-dw} 
has been found in \cite{Sobajima_MathAnn} as in Definition \ref{def:asymptotics}. 
The profiles $\{v_{\ell}(\cdot)\}_{\ell\in\Z_{\geq0}}$ are 
successively derived by the following procedure.
Put $v_{0,0}(t)=e^{-tA}(u_0+u_1)$ and 
define $\{v_{\ell,0}\}_{\ell\in\N}$ 
successively via the following relation
\[
\begin{cases}
v_{\ell,0}'(t)+Av_{\ell,0}(t)=-v_{\ell-1,0}'(t)
\quad\text{in}\ \R_+,
\\
v_{\ell,0}(0)=(-1)^{\ell}u_1,
\end{cases}
\]
or directly,
\[
v_{\ell,0}(t)
=
(-1)^{\ell}
\left[
\sum_{j=0}^\ell
\begin{pmatrix}\ell\\j\end{pmatrix}
\frac{(-tA)^j}{j!}e^{-tA}(u_0+u_1)
-
\sum_{j=0}^{\ell-1}
\begin{pmatrix}\ell-1\\j\end{pmatrix}
\frac{(-tA)^j}{j!}e^{-tA}u_0
\right]
\quad(\ell\in\N).
\]
This provides that
for every $n\in\N$, 
the solution $u$ of \eqref{eq:abst-dw} 
and the unique solution 
$U_{n}$
of 
\begin{align}\label{eq:err-equation}
\begin{cases}
U_{n}''(t)
+AU_{n}(t)
+U_{n}'(t)
=-v_{n-1,0}'(t)
\quad \text{in}\ \R_+, 
\\
(U_{n},U_{n}')(0)=(0,(-1)^n u_1)
\end{cases}
\end{align}
are connected with the relation
\begin{align}\label{decomposition}
u(t)
=\sum_{\ell=0}^{n-1} \frac{d^\ell}{dt^{\ell}}v_{\ell,0}(t)
+\frac{d^n}{dt^n}U_{n}(t). 
\end{align}
In this connection, we define the family $\{v_\ell(\cdot)\}_{\ell\in \Z_{\geq 0}}$
by $v_{\ell}=\frac{d^\ell}{dt^\ell}v_{\ell,0}$ $(\ell\in\Z_{\geq 0})$. 
In \cite{Sobajima_MathAnn}, 
it is proved that 
if  the initial data are regular ($(u_0,u_1)\in D(A^{n-\frac{1}{2}})\times D(A^{n-\frac{1}{2}})$),
then this family actually gives
the asymptotic profiles of the solution $u$ to \eqref{eq:abst-dw}. 
The precise assertion is the following. 

\begin{theorem}[{\cite[Theorem 1.1]{Sobajima_MathAnn}}]
\label{thm:regular-ini}
For every $n\in \N$ $(n\neq 0)$, 
there exists a positive constant $C_{7,n}$ such that 
the following assertion holds:
If $(u_0,u_1)\in D(A^{n-\frac{1}{2}})\times D(A^{n-\frac{1}{2}})$, 
then the corresponding solution $u$ of \eqref{eq:abst-dw} 
satisfies 
\begin{align*}
\|u(t)-V_n(t)\|
&\leq C_{7,n}(1+t)^{-(n-\frac{1}{2})}\|(u_0,u_1)\|_{D(A^{n-1/2})\times D(A^{n-1/2})},
\quad t\geq 0, 
\end{align*}
where the family $\{V_n(\cdot)\}_{n\in \N}$ is given in Definition \ref{def:asymptotics}. 
\end{theorem}
\begin{remark}
One can see that 
the asymptotic profiles $\{v_\ell(\cdot)\}_{\ell\in \Z_{\geq 0}}$ satisfy 
\[
\|v_\ell(t)\|\leq C_\ell t^{-\ell}\|(u_0,u_1)\|_{H\times H}, \quad t\geq 1
\]
with some constants $C_\ell$ depending only on $\ell$. 
Therefore Theorem \ref{thm:regular-ini} actually 
provides a justification of asymptotic expansion. 
\end{remark}

Here we would give an alternative proof via the use of the quantity $\|\cdot\|_{\mathcal{E}^\sharp}$. 
This is not a novelty, but the discussion is slightly shortened.

\begin{proof}[Sketch of the proof of Theorem \ref{thm:regular-ini}]
Let $n\in\N$.
To derive the estimate for $\frac{d^n}{dt^n}U_{n}(t)$, 
we start with the usual energy estimate for $U_{n}$
and then proceed a similar argument as 
the proof of Lemma \ref{lem:regularity1}. 

Put for $\ell,k\in \Z_{\geq 0}$, $v_{0,k}(t)=(-1)^{k}e^{-tA}(u_0+u_1)$ and 
\begin{align*}
v_{\ell,k}(t)
&=
(-1)^{\ell+k}
\left[
\sum_{j=0}^\ell
\begin{pmatrix}\ell+k\\j+k\end{pmatrix}
\frac{(-tA)^j}{j!}e^{-tA}(u_0+u_1)
-
\sum_{j=0}^{\ell-1}
\begin{pmatrix}\ell+k-1\\j+k\end{pmatrix}
\frac{(-tA)^j}{j!}e^{-tA}u_0
\right],
\end{align*}
which satisfies $\frac{d^k}{dt^k}v_{\ell,0}=A^kv_{\ell,k}$. 
Using \eqref{eq:err-equation}, we can find 
a couple of equalities
\begin{align*}
\frac{d}{dt}
E(U_{n};t)
+2\|U_{n}'(t)\|^2
&=
-2\LR{Av_{n-1,1}(t),U_{n}'(t)},
\\
\frac{d}{dt}\|U_{n}(t)\|_{\mathcal{E}^\sharp}^2
+
E(U_{n};t)
&=
-2\left\langle Av_{n-1,1}(t), U_{n}'(t)\right\rangle
-\left\langle A^{1/2}v_{n-1,1}(t), A^{1/2}U_{n}(t)\right\rangle.
\end{align*}
Therefore by a straightforward calculation, we have
\begin{align*}
&\frac{d}{dt}
\left[
\|U_{n}(t)\|_{\mathcal{E}^\sharp}^2
+
\frac{4+t}{2}E(U_{n};t)
\right]
+
E(U_{n};t)
+(4+t)\|U_{n}'(t)\|^2
\\
&\leq 2\|Av_{n-1,1}(t)\|\|U_{n}'(t)\|+\|A^{1/2}v_{n-1,1}(t)\|\|A^{1/2}U_{n}(t)\|
\\
&\quad+
\frac{1}{2}E(U_{n};t)
+(4+t)\|Av_{n-1,1}(t)\|\|U_{n}'(t)\|
\\
&\leq 
\|A^{1/2}v_{n-1,1}(t)\|^2+\frac{12+t}{2}\|Av_{n-1,1}(t)\|^2
+\frac{3}{4}E(U_{n};t)
+\frac{4+t}{2}\|U_{n}'(t)\|^2
\end{align*}
and therefore 
\begin{align*}
&\frac{d}{dt}
\left[
\|U_{n}(t)\|_{\mathcal{E}^\sharp}^2
+
\frac{4+t}{2}E(U_{n};t)
\right]
+
\frac{1}{4}
E(U_{n};t)
+\frac{4+t}{2}\|U_{n}'(t)\|^2
\\
&\leq \|A^{1/2}v_{n-1,1}(t)\|^2+\frac{12+t}{2}\|Av_{n-1,1}(t)\|^2.
\end{align*}
Since the right-hand side of the above inequality is integrable on $(0,\infty)$ 
via the estimate 
\[
\int_{0}^\infty(1+t)^{m}\|A^{\frac{m+1}{2}}e^{-tA}f\|^2\,dt\leq C_m\|f\|_{D(A^{m/2})}^2\quad (m\in\Z_{\geq0}, \ f\in D(A^{m/2}))
\]
for some positive constant $C_m$ (see e.g., \cite[Lemma 2.1]{Sobajima_MathAnn}), 
we arrive at 
\begin{align*}
&\|U_{n}(t)\|_{\mathcal{E}^\sharp}^2
+
(1+t)E(U_{n};t)
\\
&+\int_0^t\|A^{1/2}U_{n}(t)\|^2+
(1+s)\|U_{n}'(s)\|^2\,ds
\leq 
C(\|u_0+u_1\|_{D(A^{1/2})}^2+\|u_1\|_{D(A^{1/2})}^2).
\end{align*}
We can proceed a similar argument as Lemma \ref{lem:regularity1} 
(but with inhomogeneous terms)
for $w=\frac{d^m}{dt^m}U_{n}$ with $m=1,\ldots,n-1$ 
by using the respective equations 
\begin{align*}
w''(t)
+Aw(t)
+w'(t)
=-\frac{d^{m+1}}{dt^{m+1}}v_{n-1,0}(t) =-A^{m+1} v_{n-1,m+1}(t), 
\quad t\in \R_+.
\end{align*}
Then (after some computation) we finally reach 
\begin{align*}
(1+t)^{2n-1}\|u(t)-V_n(t)\|^2
&=
(1+t)^{2n-1}\left\|\tfrac{d^{n}}{dt^n}U_n(t)\right\|^2
\\	
&\leq 
(1+t)^{2n-2}
\Big(\|\tfrac{d^{n-1}}{dt^{n-1}}U_{n}(t)\|_{\mathcal{E}^\sharp}^2
+(1+t)E(\tfrac{d^{n-1}}{dt^{n-1}}U_{n};t)
\Big)
\\
&\leq C
\big(\|u_0\|_{D(A^{n-1/2})}^2+\|u_1\|_{D(A^{n-1/2})}^2\big)
\end{align*}
which is the desired inequality.
\end{proof}

\begin{remark}\label{rem:comment}
The above procedure also provides the following inequality
\[
(1+t)^{2n}\Big(\|u(t)-V_n(t)\|_{\mathcal{E}^\sharp}^2
+(1+t)E(u-V_n;t)
\Big)
\leq \widetilde{C}_{7,n}
\big(\|u_0\|_{D(A^{n})}^2+\|u_1\|_{D(A^{n})}^2\big), \quad t\geq 0.
\]
\end{remark}

\subsubsection{Asymptotic expansion for general initial data}

In contrast, 
\rev{%
Theorem \ref{intro:thm:abst-expansion}
}
asserts the same consequence 
as that of Theorem \ref{thm:regular-ini}
without any additional regularity for the initial data. 
To prove Theorem \ref{intro:thm:abst-expansion}, 
the following decomposition of identity $1_H$
via the resolvent operator $J=(1_H+A)^{-1}$ is crucial.

\begin{lemma}\label{lem:unit-decomposition}
For every $n\in\N$, there exist
a pair of linear bounded operators $(I_n, J_n)$ 
from $H$ to $D(A^n)$ such that $1_H=I_n+A^{n}J_n$.
\end{lemma}
\begin{proof}
Observing that 
\begin{align*}
I_H&=(1_H+A)^{2n-1}J^{2n-1}
\\
&=
\sum_{k=0}^{n-1}
\begin{pmatrix}2n-1\\k\end{pmatrix}A^{k}J^{2n-1}
+
\sum_{k=n}^{2n-1}
\begin{pmatrix}2n-1\\k\end{pmatrix}A^{k}J^{2n-1}
\\
&=
J^{n}
\left(
\sum_{k=0}^{n-1}
\begin{pmatrix}2n-1\\k\end{pmatrix}A^{k}J^{n-1}\right)
+
A^{n}J^n
\left(\sum_{k=0}^{n-1}
\begin{pmatrix}2n-1\\n+k\end{pmatrix}A^{k}J^{n-1}\right), 
\end{align*}
we can choose $I_n$ and $J_n$ respectively as
\[
I_n=J^{n}
\left(
\sum_{k=0}^{n-1}
\begin{pmatrix}2n-1\\k\end{pmatrix}A^{k}J^{n-1}\right), 
\quad 
J_n=
J^n
\left(\sum_{k=0}^{n-1}
\begin{pmatrix}2n-1\\n+k\end{pmatrix}A^{k}J^{n-1}\right).
\]
\end{proof}

By using $I_n$ and $J_n$, we decompose 
the solution $u$ of \eqref{eq:abst-dw} as 
\begin{equation}\label{eq:dec-m}
u(t)=I_nu(t)+A^{n}J_nu(t).
\end{equation}
Then we can see that 
the decomposition in \cite{Sobajima_MathAnn} 
(explained in the beginning of Section \ref{subsec:pre})
is applicable to the term $I_nu$. 
\begin{lemma}\label{lem:error-profile}
Let $u$ be the solution of \eqref{eq:abst-dw} with $(u_0,u_1)\in\mathcal{H}$ 
and let $n\in\N$.
Put $U_{*}$ as the solution of
the following problem:
\begin{align*}
\begin{cases}
U_{*}''(t)+AU_{*}(t)+U_{*}'(t)=-I_nv_{n-1,0}'(t)\quad \text{in}\ \R_+, 
\\
(U_{*},U_{*}')(0)=(0,(-1)^{n}I_nu_1),
\end{cases}
\end{align*}
where $v_{n-1,0}$ is given in the beginning of Section \ref{subsec:pre}.
Then one has 
\[
I_nu(t)=
\sum_{\ell=0}^{n-1}I_n v_{\ell}(t)+\frac{d^n}{dt^n}U_{*}(t)
=I_nV_n(t)+\frac{d^n}{dt^n}U_{*}(t).
\]
\end{lemma}

\begin{proof}
Since $u_*=I_nu$ satisfies the following problem
(especially, the regularity of the initial data):
\begin{gather*}
\begin{cases}
u_{*}''(t)+Au_{*}(t)+u_{*}'(t)=0 \quad \text{in}\ \R_+, 
\\
(u_{*},u_{*}')(0)=(I_n u_0,I_nu_1)
\in D(A^{n+1/2})\times D(A^{n}),
\end{cases}
\end{gather*}
the decomposition in \cite{Sobajima_MathAnn}
is applicable, which is the desired assertion.
\end{proof}

\begin{proof}[Proof of Theorem \ref{intro:thm:abst-expansion}]
We employ Lemmas 
\ref{lem:unit-decomposition} 
and \ref{lem:error-profile} to 
the weak solution $u$ of \eqref{eq:abst-dw}. Then 
\begin{align*}
u(t)-V_n(t)
&=
I_nu(t)+A^{n}J_nu(t)
-I_nV_n(t)-A^{n}J_nV_n(t)
\\
&=
\frac{d^n}{dt^n}U_{*}(t)+A^{n}J_nu(t)+A^nJ_nV_n(t).
\end{align*}
Hereafter we always consider estimates for 
$t\geq 1$. 
Applying the procedure of Theorem \ref{thm:regular-ini} 
and also Remark \ref{rem:comment}, we can deduce 
\begin{align*}
\left\|\frac{d^n}{dt^n}U_{*}(t)\right\|^2
+
t
E\left(\frac{d^n}{dt^n}U_{*};t\right)
&\leq 
\widetilde{C}_{7,n}t^{-2n}
\|(I_nu_0,I_nu_1)\|_{D(A^{n})\times D(A^{n})}^2. 
\end{align*}
Moreover, since 
$J_nu$ satisfies $(J_nu,J_nu')(0)=(J_nu_0,J_nu_1)\in D(\mathcal{L}^{2n})$,
Lemma \ref{lem:regularity1} gives 
for $t\geq 1$, 
\begin{align*}
\left\|A^{n}J_nu(t)\right\|^2
+
t
E(A^{n}J_nu;t)
&\leq C_{6,2n}t^{-2n}
\|(J_nu_0,J_nu_1)\|_{D(\mathcal{L}^{2n})}^2.
\end{align*}
Furthermore, 
the boundedness of $(-tA)^ke^{-tA}$ in $t$ (for all $k$)
gives for $\ell=0,\ldots, n-1$, 
\begin{align*}
\|A^{n}J_n v_\ell(t)\|^2
+t
E(A^{n}J_nv_\ell;t)
\leq 
Ct^{-2n}\|(u_0,u_1)\|_{H\times H}^2
\end{align*}
for some positive constant $C$ (depending only on $m$ and $\ell$). 
These yield a reasonable estimate for $A^{n}J_n V_n(t)$. 
Combining these inequalities, we arrive at 
the desired inequality. 
The proof is complete.
\end{proof}

\section{Application to the damped wave equation}\label{sec:damped}

In this section, 
we consider the following initial-boundary value problem of 
the $N$-dimensional damped wave equation 
\begin{equation}\label{eq:Nd-ext-dw}
\begin{cases}
\pa_t^2u(x,t)-\Delta u(x,t)+\pa_t u(x,t) =0 
&\text{in}\ \Omega \times \R_+, 
\\
u(x,t) =0 
&\text{on}\ \pa\Omega \times \R_+, 
\\
(u,\pa_tu)(x,0)=(u_0(x),u_1(x))
&\text{in}\ \Omega, 
\end{cases}
\end{equation}
where $\Omega$ is either the whole space $\R^N$ or 
an exterior (connected) domain with a smooth boundary $\pa\Omega$.
We shall discuss several properties of 
weak solutions to \eqref{eq:Nd-ext-dw} 
as an application of the asymptotic expansion in Theorem \ref{intro:thm:abst-expansion}
under the assumption 
\begin{equation}\label{ass:ini}
(u_0,u_1)\in \mathbf{E}(\Omega)\cap \mathbf{L}^q(\Omega)
=(H_0^1(\Omega)\cap L^q(\Omega))\times (L^2(\Omega)\cap L^q(\Omega)) 
\end{equation}
with $q\in [1,2)$. Here we fix 
\[
v_{00}(x)=u_0(x)+u_1(x), \quad x\in\Omega.
\] 
For the Cauchy problem ($\Omega=\R^N$ and ignoring the boundary condition in \eqref{eq:Nd-ext-dw}), 
we use $\Delta=\Delta_{\R^N}$ (endowed with domain $D(\Delta)=H^2(\R^N)$) and the usual heat semigroup 
\[
e^{t\Delta}f(x)
=
e^{t\Delta_{\R^N}}f(x)
=
(4\pi t)^{-N/2}\int_{\R^N}e^{-\frac{|x-y|^2}{4t}}f(y)\,dy, \quad x\in \R^N.
\]
For the exterior problem, 
we use the Dirichlet heat semigroup $\{e^{t\Delta_{\Omega}}\}_{t\geq 0}$ 
which is generated by the Dirichlet Laplacian $\Delta_{\Omega}$
endowed with domain $D(\Delta_{\Omega})=H^2(\Omega)\cap H_0^1(\Omega)$, 
and we assume $0\notin \overline{\Omega}$ and 
set positive constants $r_\Omega$ and $R_\Omega$ as 
$r_\Omega=\min_{\pa\Omega}|x|$ and $R_\Omega=\max_{\pa\Omega}|x|$, that is, 
$B(0,r_\Omega)\subset \Omega^c\subset B(0,R_\Omega)$
(for the convenience put $r_{\R^N}=0$, $R_{\R^N}=\infty$). 
We discuss the large-time behavior of weak solution $u$ and also the one 
for the energy functional 
\[
E_{\Omega}(u;t)=\int_{\Omega}\Big(|\pa_tu(x,t)|^2+|\nabla u(x,t)|^2\Big)\,dx.
\]
Here we again use the same notations as in 
Definition \ref{def:asymptotics} with $A=-\Delta_{\Omega}$, that is, 
\begin{align*}
v_0(t)
&=e^{t\Delta_\Omega}v_{00}, 
\\
v_\ell(t)
&=
(-\Delta_\Omega)^\ell\left(
\sum_{j=0}^\ell
\begin{pmatrix}
2\ell\\\ell+j
\end{pmatrix}\frac{(t\Delta_\Omega)^j}{j!}e^{t\Delta_\Omega}v_{00}
-
\sum_{k=0}^{\ell-1}
\begin{pmatrix}
2\ell-1\\ \ell+k
\end{pmatrix}\frac{(t\Delta_\Omega)^k}{k!}e^{t\Delta_\Omega}u_0
\right), \quad \ell\in \N
\end{align*}
and also $V_0(t)=0$ and $V_m(t)=\sum_{\ell=0}^{m-1}v_{\ell}(t)$ for $\ell\in\N$.

\subsection{Decay estimates for $L^2$-norm and the energy functional}
The comparison principle shows that 
$|e^{t\Delta_\Omega}f|\leq e^{t\Delta_{\R^N}}|f_*|$ for any $f\in L^2(\Omega)$, 
where $f_*:\R^N\to \R$ is the zero-extension of $f$. 
This immediately gives 
the usual $L^2$-$L^q$ estimate 
\[
\|e^{t\Delta }f\|_{L^2(\Omega)}\leq Ct^{-\frac{N}{2}(\frac{1}{q}-\frac{1}{2})}
\|f\|_{L^q(\Omega)}, \quad t>0
\]
for $f\in L^q(\Omega)$ with $q\in [1,2]$. 
Therefore under the condition \eqref{ass:ini}, 
we additionally have the decay property
\[
\rev{%
\|e^{t\Delta_\Omega}v_{00}\|_{L^2(\Omega)}
}
+\|e^{t\Delta_\Omega}u_0\|_{L^2(\Omega)}
\leq 
Ct^{-\frac{N}{2}(\frac{1}{q}-\frac{1}{2})}\|(u_0,u_1)\|_{\mathbf{L}^q(\Omega)}.
\]
In this case, we can also find 
the respective decay properties for 
$u-V_m$ $(m\in \N)$. 
These can be formulated as the following. 

\begin{proposition}\label{prop:dw-whole-diffusionphenomena}
For every $q\in [1,2)$ and $m\in \Z_{\geq 0}$, 
there exists a positive constant $C_{8,q,m}$ such that 
the following assertion holds:
If $(u_0,u_1)\in \mathbf{E}(\Omega)\cap \mathbf{L}^q(\Omega)$, 
then the corresponding weak solution $u$ of \eqref{eq:Nd-ext-dw} 
satisfies  
\begin{gather}
\label{eq:lem:diffphen}
\|u(t)-V_m(t)\|_{L^2(\Omega)}^2+tE_{\Omega}(u-V_m;t)
\leq C_{8,q,m}
t^{-N(\frac{1}{q}-\frac{1}{2})-2m}
\|(u_0,u_1)\|_{\mathbf{E}(\Omega)\cap \mathbf{L}^q(\Omega)}^2, 
\quad t\geq 1, 
\end{gather}
where the family $\{V_m(\cdot)\}_{m\in\Z_{\geq 0}}$ is given 
in Definition \ref{def:asymptotics} with $A=-\Delta_{\Omega}$.
\end{proposition}
\begin{proof}
Fix $m_0\in\N$ satisfying 
\[
\frac{N}{2}\left(\frac{1}{q}-\frac{1}{2}\right)< m_0
\leq \frac{N}{2}\left(\frac{1}{q}-\frac{1}{2}\right) +1.
\]
Applying Theorem \ref{intro:thm:abst-expansion} with $n=m+m_0$, we have
\begin{align*}
\left\|
u(t)-V_{m+m_0}(t)
\right\|_{L^2(\Omega)}^2
+
tE_{\Omega}\left(u-V_{m+m_0};t\right)
&\leq C_{1,m+m_0}t^{-2m_0-2m}\|(u_0,u_1)\|_{\mathbf{E}(\Omega)}^2
\\
&\leq C_{1,m+m_0}t^{-N(\frac{1}{q}-\frac{1}{2})-2m}\|(u_0,u_1)\|_{\mathbf{E}(\Omega)}^2.
\end{align*}
On the one hand, note that  $V_{m+m_0}-V_{m}=\sum_{\ell=m}^{m+m_0-1}v_\ell$ 
and for any $\ell=m,m+1, \cdots,m+m_0-1$, we have
\begin{align*}
\left\|
v_\ell(t)
\right\|_{L^2(\Omega)}^2
+
tE_{\Omega}\left(v_\ell;t\right)
&\leq Ct^{-2\ell}
\Big(
\|e^{\frac{t}{2}\Delta_{\Omega}}v_{00}\|_{L^2(\Omega)}^2
+\|e^{\frac{t}{2}\Delta_{\Omega}}u_0\|_{L^2(\Omega)}^2
\Big)
\\
&\leq Ct^{-N(\frac{1}{q}-\frac{1}{2})-2\ell}
\Big(\|(u_0,u_1)\|_{\mathbf{L}^q(\Omega)}^2\Big).
\end{align*}
Combining these estimates, we can reach the desired inequality.
\end{proof}

\subsubsection{Optimality of decay rates for $\Omega=\R^N$ $(N\in\N)$}
Proposition \ref{prop:dw-whole-diffusionphenomena} with $m=0$ actually provides
the optimal decay rate for the $L^2$-norm and the energy
of weak solutions with general initial data 
except the case of two-dimensional exterior domain, 
which will be discussed later.

The following proposition gives the optimality of the decay rate 
for the case $\Omega=\R^N$ $(N\in \N)$.

\begin{proposition}\label{prop:optimality1}
Let $\Omega=\R^N (N\in\N)$ 
and let $u$ be the weak solution of \eqref{eq:Nd-ext-dw} 
with $(u_0,u_1)\in \mathbf{E}(\R^N)\cap \mathbf{L}^1(\R^N)$. 
Then the following assertions hold:
\begin{itemize}
\item[\bf (i)]If 
$\int_{\R^N}(u_0+u_1)\,dx\neq 0$, 
then 
the following quantities are both finite and strictly positive:
\[
\limsup_{t\to\infty} \Big(t^{\frac{N}{2}}\|u(t)\|_{L^2(\R^N)}^2\Big), 
\quad 
\limsup_{t\to\infty} \Big(t^{\frac{N}{2}+1}E_{\R^N}(u;t)\Big).
\]
\item[\bf (ii)]If $u_1=-u_0$ and 
$\int_{\R^N}u_0\,dx\neq 0,$
then 
the following quantities are both finite and strictly positive:
\[
\limsup_{t\to\infty} \Big(t^{\frac{N}{2}+2}\|u(t)\|_{L^2(\R^N)}^2\Big), 
\quad 
\limsup_{t\to\infty} \Big(t^{\frac{N}{2}+3}E_{\R^N}(u;t)\Big).
\]
\end{itemize}
\end{proposition}

To prove this, we use the following Sobolev-type inequality. 

\begin{lemma}\label{lem:Nash-GN}
There exists a positive constant $C_9$ such that 
for every $f\in H_0^1(\Omega)\cap L^1(\Omega)$, 
\[
\|f\|_{L^{1+\frac{2}{N}}(\Omega)}^{1+\frac{2}{N}}
\leq 
C_9\|f\|_{L^{1}(\Omega)}^{\frac{N^2+4}{N(N+2)}}
\|\nabla f\|_{L^{2}(\Omega)}^{\frac{4}{N+2}}.
\]
\end{lemma}
\begin{proof}
If $N=2$, then
this is nothing but the Nash inequality 
\begin{equation}
\label{eq:nash}
\|f\|_{L^2(\Omega)}^{2+\frac{4}{N}}\leq C_{\rm Nash}\|f\|_{L^1(\Omega)}^\frac{4}{N}\|\nabla f\|_{L^2(\Omega)}^2.
\end{equation}
If $N\geq 3$,
then we just adapt \eqref{eq:nash} with the H\"older inequality 
$
\|f\|_{L^{1+\frac{2}{N}}(\Omega)}\leq 
\|f\|_{L^1(\Omega)}^{\frac{N-2}{N+2}}\|f\|_{L^2(\Omega)}^{\frac{4}{N+2}}$.
If $N=1$, then we combine \eqref{eq:nash} with the Gagliardo--Nirenberg inequality of the form
$\|f\|_{L^{3}(\R)}
\leq C_{\rm GN}\|f\|_{L^2(\R)}^{\frac{5}{6}}
\|f'\|_{L^2(\R)}^{\frac{1}{6}}$.
The proof is complete.
\end{proof}

\begin{proof}[Proof of Proposition \ref{prop:optimality1}]
{\bf (i)}
First note that Proposition \ref{prop:dw-whole-diffusionphenomena} with 
$(q,m)=(1,0),(1,1)$ gives that 
\[
t^{\frac{N}{2}}\|u(t)\|_{L^2(\R^N)}^2
+t^{\frac{N}{2}+1}E_{\R^N}(u;t)
\leq C_{8,1,0}
\|(u_0,u_1)\|_{\mathbf{E}(\R^N)\cap \mathbf{L}^1(\R^N)}^2,
\quad t\geq 1
\]
and 
\[
t^{\frac{N}{2}+1}E_{\R^N}(u-v_0;t)
\leq C_{8,1,1}t^{-2}
\|(u_0,u_1)\|_{\mathbf{E}(\R^N)\cap \mathbf{L}^1(\R^N)}^2.
\quad t\geq 1,
\]
where $v_0(t)=e^{t\Delta_{\R^N}}(u_0+u_1)$. 
The first inequality assures that both quantities in the assertion is finite.
Moreover, we see that 
\[
t^{\frac{N}{2}+1}\|\nabla v_0(t)\|_{L^2(\R^N)}^2
\leq 
2t^{\frac{N}{2}}\|v_0(t/2)\|_{L^2(\R^N)}^2.
\]
Therefore it suffices to show only the lower bound for $\|\nabla v_0(t)\|_{L^2(\R^N)}$. 
Here we employ the strategy of 
the test function method in Ikeda--Sobajima \cite{IkSo2019NA}. 
Fix a non-increasing function $\eta\in C^\infty(\R)$ satisfying 
$\mathbbm{1}_{(-\infty,\frac{1}{2}]}\leq \eta\leq \mathbbm{1}_{(-\infty,1]}$ 
and set 
for the parameter $\rho>0$, 
\[
\zeta_\rho(x,t)=\eta\big(\xi_\rho(x,t)\big), \quad \xi_\rho(x,t)=\frac{|x|^2+t}{\rho},
\]
where $\mathbbm{1}_D$ denotes the indicator function on $D$.
Since the dominated convergence theorem gives 
\[
\lim_{\rho\to \infty}\int_{\R^N}(u_0+u_1)\zeta_\rho\,dx
=\int_{\R^N}(u_0+u_1)\,dx\neq 0, 
	\]
we can choose a positive constant $\rho_*$ such that for every $\rho\geq \rho_*$, 
\[
\left|\int_{\R^N}(u_0+u_1)\zeta_\rho\,dx\right|\geq c_0=\frac{1}{2}\left|\int_{\R^N}(u_0+u_1)\,dx\right|>0.
\]
Multiplying the equation $\pa_tv_0-\Delta v_0=0$ to $\zeta_{\rho}$ 
and integrating it over $\R^N$, we see 
from integration by parts that 
\begin{align*}
0=
\int_{\R^N}(\pa_tv_0-\Delta v_0)\zeta_\rho\,dx
=\frac{d}{dt}\int_{\R^N}v_0\zeta_\rho\,dx
-\int_{\R^N}v_0(\pa_t \zeta_\rho+\Delta \zeta_\rho)\,dx.
\end{align*}
Integrating it again over $\R_+$ 
and using the H\"older inequality 
(and putting $\mathcal{Q}=\R^N\times \R_+$), we 
have
\begin{align*}
\left|\int_{\R^N}(u_0+u_1)\zeta_\rho\,dx\right|^{1+\frac{2}{N}}
&=\left|\iint_{\mathcal{Q}}v_0(\pa_t \zeta_\rho+\Delta \zeta_\rho)\,dxdt\right|^{1+\frac{2}{N}}
\\
&\leq 
\left(\iint_{\mathcal{Q}}|\pa_t \zeta_\rho+\Delta \zeta_\rho|^{\frac{N}{2}+1}\,dxdt\right)^{\frac{2}{N}} \rho Y'(\rho), 
\end{align*}
where we have set
\[
Y(T)=\int_{0}^T\left(\iint_{\mathcal{Q}}|v_0|^{1+\frac{2}{N}}\mathbbm{1}_{[\frac{1}{2},1]}(\xi_\rho)\,dxdt\right)\,\frac{d\rho}{\rho}
\quad
\left(\leq \log 2\int_0^T\|v_0(t)\|_{L^{1+\frac{2}{N}}(\R^N)}^{1+\frac{2}{N}}\,dt\right).
\]
Noting that 
\[
\iint_{\mathcal{Q}}|\pa_t \zeta_\rho+\Delta \zeta_\rho|^{\frac{N}{2}+1}\,dxdt\leq C
\]
for some positive constant $C$ independent of $\rho$, 
we arrive at the following lower bound 
for $Y$ with the logarithmic function:
\[
Y(\rho_*)
+\frac{c_0^{1+\frac{2}{N}}}{C^{\frac{2}{N}}}\log \frac{T}{\rho_*}
\leq 
Y(T)
\leq 
\log 2\int_0^T\|v_0(t)\|_{L^{1+\frac{2}{N}}(\R^N)}^{1+\frac{2}{N}}\,dt, \quad T>\rho_*.
\]
Applying Lemma \ref{lem:Nash-GN}, we deduce
\[
\int_0^T\|\nabla v_0(t)\|_{L^2(\R^N)}^{\frac{4}{N+2}}\,dt
\geq \delta\log T,
\quad T>\rho_*
\]
for some positive constant $\delta>0$.
This gives
$\limsup\limits_{t\to \infty}\big(t^{\frac{N}{2}+1}\|\nabla v_0(t)\|_{L^2(\R^N)}^2\big)>0$. 

{\bf (ii)} 
In this case, the first asymptotic profile vanishes, that is, $v_0(t)=e^{t\Delta}(u_0+u_1)=0$. 
Proposition \ref{prop:dw-whole-diffusionphenomena} 
with $(q,m)=(1,1),(1,2)$ gives that 
\[
t^{\frac{N}{2}+2}\|u(t)\|_{L^2(\R^N)}^2
+t^{\frac{N}{2}+3}E_{\R^N}(u;t)
\leq C_{8,1,1}
\|u_0\|_{H_0^1(\Omega)\cap L^1(\Omega)}^2, 
\quad t\geq 1
\]
and 
\[
t^{\frac{N}{2}+3}
E_{\R^N}(u-v_1;t)
\leq C_{8,1,2}t^{-2}\|u_0\|_{H_0^1(\Omega)\cap L^1(\Omega)}^2, 
\quad t\geq 1
\]
and therefore the quantities in the assertion are both finite.
Observe that the second asymptotic profile of $u$ is given by $v_1(t)=\Delta e^{t\Delta }u_0$. 
As in the proof of {\bf (i)}, we can find
\[
\limsup_{t\to\infty} \Big(t^{\frac{N}{2}+1}\|\nabla e^{t\Delta}u_0\|_{L^2(\R^N)}^2\Big)>0.
\]
Using the interpolation estimate 
$\|\nabla f\|_{L^2(\R^N)}^3\leq \|f\|_{L^2(\R^N)}^2
\|\nabla \Delta f\|_{L^2(\R^N)}$ and the $L^2$-$L^1$ estimate, we also have 
\begin{align*}
\Big(t^{\frac{N}{2}+1}\|\nabla e^{t\Delta}u_0\|_{L^2(\R^N)}^2\Big)^{3}
&\leq 
\Big(t^{\frac{N}{4}}\|e^{t\Delta}u_0\|_{L^2(\R^N)}\Big)^{4}
\Big(t^{\frac{N}{2}+3}\|\nabla \Delta e^{t\Delta}u_0\|_{L^2(\R^N)}^{2}\Big)
\\
&\leq 
C\|u_0\|_{L^1(\R^N)}^{4}
\Big(t^{\frac{N}{2}+3}\|\nabla v_1(t)\|_{L^2(\R^N)}^{2}\Big).
\end{align*}
Therefore we find that 
$t^{\frac{N}{2}+3}\|\nabla v_1(t)\|_{L^2(\R^N)}^2$ 
and also 
$t^{\frac{N}{2}+2}\|v_1(t)\|_{L^2(\R^N)}^2$ are bounded below by positive constants. 
The proof is complete.
\end{proof}

\begin{remark}
The assumption $u_1=-u_0$ in Proposition \ref{prop:optimality1} {\bf (ii)} 
can be weakened in the following sense.
Roughly speaking, a similar phenomenon occurs 
when $e^{t\Delta}(u_0+u_1)$ decays faster than $\Delta e^{t\Delta}u_0$.
For example, if 
$u_0\in H^1(\R^N)\cap L^1(\R^N)$ with $\int_{\R^N}u_0\,dx\neq 0$
and 
$u_1=-u_0+\Delta^2 w_0$ for some $w_0\in H^4(\R^N)\cap L^1(\R^N)$, 
then the profiles $v_0(t)=e^{t\Delta}(u_0+u_1)$ and 
$v_1(t)=\Delta e^{t\Delta}u_0-(2+t\Delta) \Delta e^{t\Delta}(u_0+u_1)$ 
satisfy
\begin{gather*}
\|v_0(t)\|_{L^2(\R^N)}^2=O(t^{-\frac{N}{2}-4}), 
\quad 
E_{\R^N}(v_0;t)=O(t^{-\frac{N}{2}-5}),
\\
E_{\R^N}(v_1;t)=\|\nabla \Delta e^{t\Delta}u_0\|_{L^2(\R^N)}^2+O(t^{-\frac{N}{2}-5}).
\end{gather*}
However, the same argument as in the proof of Proposition \ref{prop:optimality1} 
shows that the following quantities for $V_1$ are both finite and strictly positive:
\begin{align*}
\limsup_{t\to\infty}
\Big(t^{\frac{N}{2}+2}
\|\Delta e^{t\Delta}u_0\|_{L^2(\R^N)}^2\Big), 
\quad
\limsup_{t\to\infty}
\Big(t^{\frac{N}{2}+3}\|\nabla \Delta e^{t\Delta}u_0\|_{L^2(\R^N)}^2\Big).
\end{align*}
In this case, actually we could check that 
the asymptotic behavior of $u$ is clarified as $\Delta e^{t\Delta}u_0$ which differs from the usual one $v_0(t)=e^{t\Delta}(u_0+u_1)$.
This phenomenon occurs also for the exterior problems.
\end{remark}

\subsubsection{Optimality of decay rates for the exterior problem ($N\geq 3$)}

In this case we employ the test function method 
with positive harmonic functions satisfying the Dirichlet boundary condition. 
The following lemma describes
the asymptotic behavior of 
the positive harmonic function at spatial infinity similar to 
$1-|x|^{2-N}$. 
The construction is well-known (see e.g., Jost \cite[Section 22]{JostBook}), but for the reader's convenience, 
we give a sketch of the proof.

\begin{lemma}\label{lem:harmonicsNd}
Let $\Omega$ be an exterior domain in $\R^N$ $(N\geq 3)$.
Then there exists a unique solution $h_\Omega\in C(\overline{\Omega})\cap C^\infty(\Omega)$ of the problem 
\begin{equation}\label{harmonic}
\begin{cases}
\Delta h_\Omega(x)=0, 
& x\in \Omega,
\\
h_\Omega (x)=0, 
& x\in \pa\Omega,
\\
h_\Omega(x)>0, 
& x\in \Omega
\end{cases}
\end{equation}
satisfying $\lim\limits_{|x|\to\infty} h_\Omega(x)=1$. Moreover, $h_\Omega$
satisfies 
\begin{gather*}
|\nabla h_\Omega(x)|
\leq \frac{2^{N-1}NR_{\Omega}^{N-2}}{|x|^{N-1}}, \quad x\in B(0,2R_\Omega)^c.
\end{gather*}
\end{lemma}
\begin{proof}
For $n\in\N$ satisfying $n>R_\Omega$,
take the unique minimizer $\widetilde{h}_n$ for 
the functional 
\[
\int_{B(0,n)}|\nabla h|^2\,dx
\]
restricted to 
the closed convex set $\left\{h\in H_0^1(B(0,n))\;;\;h\geq \mathbbm{1}_{\Omega^c}\right\}$.
Letting $n\to \infty$, 
we have the increasing limit $\widetilde{h}(x)=\lim_{n\to \infty}h_n(x)$
which satisfies
\[
\frac{r_\Omega^{N-2}}{|x|^{N-2}}
\leq \widetilde{h}(x) \leq
\frac{R_\Omega^{N-2}}{|x|^{N-2}}, \quad x\in \Omega.
\] This is verified 
via the Stampacchia's theorem and 
the maximum principle. 
Then we can see that 
the function $1-\widetilde{h}(x)$ satisfies \eqref{harmonic}.  
If $|x_0|\geq 2R_\Omega$, then 
the mean-value theorem for harmonic functions (see e.g., Evans \cite[Section 2.2]{EvansBook}) gives 
\[
|\nabla h_\Omega(x_0)|
\leq \frac{2N}{|x_0|}\|h_\Omega\|_{L^\infty(\pa B(x_0,\frac{|x_0|}{2}))}
\leq \frac{2^{N-1}NR_{\Omega}^{N-2}}{|x_0|^{N-1}}.
\]
The proof is complete.
\end{proof}
\begin{remark}
If $N=2$, then the above procedure gives $\widetilde{h}\equiv 1$ ($h_{\Omega}\equiv 0$) which is not meaningful. We will discuss the case of two-dimensional exterior domain later.
\end{remark}

The optimality of decay estimates in Proposition \ref{prop:dw-whole-diffusionphenomena} with $m=0,1$  
can be found via the use of $h_\Omega$. 
\begin{proposition}\label{prop:optimality3}
Let $\Omega$ be an exterior domain in $\R^N$ $(N\geq 3)$
and let $u$ be the weak solution of \eqref{eq:Nd-ext-dw} with 
$(u_0,u_1)\in \mathbf{E}(\Omega)\cap \mathbf{L}^1(\Omega)$. 
Let $h_{\Omega}$ be given in Lemma \ref{lem:harmonicsNd}. 
Then the following assertions hold: 
\begin{itemize}
\item[\bf (i)]
If 
$\int_{\Omega}(u_0+u_1)h_\Omega\,dx\neq 0$, 
then 
the following quantities are both finite and strictly positive:
\[
\limsup_{t\to\infty} \Big(t^{\frac{N}{2}}\|u(t)\|_{L^2(\Omega)}^2\Big), 
\quad 
\limsup_{t\to\infty} \Big(t^{\frac{N}{2}+1}E_{\Omega}(u;t)\Big).
\]
\item[\bf (ii)]
If $u_1=-u_0$
and 
$\int_{\Omega}u_0 h_{\Omega}\,dx\neq 0$, 
then 
the following quantities are both finite and strictly positive:
\[
\limsup_{t\to \infty} \Big(t^{\frac{N}{2}+2}\|u(t)\|_{L^2(\Omega)}^2\Big), 
\quad 
\limsup_{t\to \infty} \Big(t^{\frac{N}{2}+3}E_\Omega(u;t)\Big).
\]
\end{itemize}
\end{proposition}
\begin{proof}
We first note that by Proposition \ref{prop:dw-whole-diffusionphenomena},
all quantities are verified to be finite.
We would only explain the difference compared to the proof of Proposition \ref{prop:optimality1}.
For {\bf (i)}, fix $\eta$ as in the proof of Proposition \ref{prop:optimality1} and 
\[
\widetilde{\zeta}_\rho(x,t)=\eta\big(\widetilde{\xi}_\rho(x,t)\big), \quad \widetilde{\xi}_\rho(x,t)=\frac{[(|x|-2R_\Omega)_+]^2+t}{\rho}.
\]
Then we can check that 
\begin{align*}
|\pa_t(\widetilde{\zeta}_\rho h_\Omega)+\Delta (\widetilde{\zeta}_\rho h_\Omega)|
&\leq 
\big(|\pa_t\widetilde{\zeta}_\rho|+|\Delta \widetilde{\zeta}_\rho|\big)h_\Omega+2|\nabla \widetilde{\zeta}_\rho\cdot\nabla h_\Omega|
\leq
\frac{K}{\rho}\chi(\widetilde{\xi}_\rho)h_{\Omega}
\end{align*}
for some positive constant $K$. 
Using the above estimate, we can proceed the test function method 
with the functional
\[
Y_{\Omega}(T)=\int_{0}^T\left(\iint_{\Omega\times(0,\infty)}|v_0|^{1+\frac{2}{N}}\mathbbm{1}_{[\frac{1}{2},1]}(\widetilde{\xi}_\rho)h_{\Omega}\,dxdt\right)\,\frac{d\rho}{\rho}.
\]
Using Lemma \ref{lem:Nash-GN} with $h_\Omega\leq 1$, we could find the lower bound of 
$t^{\frac{N}{2}+1}\|\nabla v_0(t)\|_{L^2(\Omega)}^2$
and also 
$t^{\frac{N}{2}}\|v_0(t)\|_{L^2(\Omega)}^2$.
The proof of {\bf (ii)} is 
completely the same as the one of Proposition \eqref{prop:optimality1}. 
\end{proof}

\subsubsection{Decay estimates for the two-dimensional exterior problem}
In the two-dimensional case, 
the series of estimates in Proposition \ref{prop:dw-whole-diffusionphenomena} 
can be improved by restricting 
the choice of initial data.
To describe this, we use {\it unbounded}  
positive harmonic functions satisfying the Dirichlet boundary condition, which is the reason why the two-dimensional case 
is exceptional.

\begin{lemma}\label{lem:harmonics2d}
If $Omega$ is a two-dimensional exterior domain, 
then there exists a unique solution $h_\Omega\in C^\infty(\Omega)$ of the problem 
\begin{equation}\label{harmonic2d}
\begin{cases}
\Delta h_\Omega(x)=0, 
& x\in \Omega,
\\
h_\Omega (x)=0, 
& x\in \pa\Omega,
\\
h_\Omega(x)>0, 
& x\in \Omega
\end{cases}
\end{equation}
satisfying $\log h_B\left(\frac{|x|}{R_\Omega}\right)
\leq 
h_\Omega(x)
\leq 
\log \left(\frac{|x|}{r_\Omega}\right)$ and 
$|x|^2\nabla h_\Omega(x)-x\in L^\infty(\Omega;\R^2)$.  
In particular, 
there exists a positive constant $R_\Omega^\sharp (>R_\Omega)$ 
such that 
for every $x\in B(0,R_\Omega^\sharp)^c$, 
\[
\left|\frac{h_\Omega(x)}{\log |x|}-1\right|
+
\left|x\cdot \nabla h_\Omega(x)-1\right|\leq \frac{1}{2}.
\]
\end{lemma}
\begin{proof}
The function $h_\Omega$ can be constructed 
via the Kelvin transform 
$K:y\mapsto \frac{y}{|y|^2}$
and the Green function on the set $\{K(x)\;;\;x\in \Omega\}\cup \{0\}$. 
\end{proof}
Hereafter we shall use the non-degenerate positive harmonic function
\[
H_\Omega(x)=1+h_\Omega(x),\quad x\in \Omega,
\] 
where $h_\Omega$ is given in Lemma \ref{lem:harmonics2d}. 
Then we introduce the weighted $L^1$-space
\[
L_{d\mu}^1
=\{f\in L^1(\Omega)\;;\;H_\Omega f\in L^1(\Omega)\}, 
\quad 
\|f\|_{L_{d\mu}^1}=\|H_\Omega f\|_{L^1(\Omega)};
\]
note that $f\in L_{d\mu}^1$ is equivalent to $(1+\log \frac{|x|}{r_{\Omega}})f\in L^{1}(\Omega)$ 
which condition appears in Introduction.
We list several properties of 
the Dirichlet heat semigroup $\{e^{t\Delta_{\Omega}}\}_{t\geq 0}$
on $L^2(\Omega)$ 
(see e.g., \cite[Appendix]{ISTWpre}).
\begin{lemma}\label{lem:2dexteriorheat}
The following inequalities hold:
\begin{itemize}
\item[\bf (i)]For every $f\in L^2(\Omega)\cap L_{d\mu}^1$, 
\[
\|e^{t\Delta_{\Omega}}f\|_{L_{d\mu}^1}
\leq 
\|f\|_{L_{d\mu}^1},\quad t\geq 0.
\]
\item[\bf (ii)]There exists a positive constant $C_{10}$ such that 
for every $f\in L^2(\Omega)$,  
\begin{align*}
\left\|H_\Omega^{-1}e^{t\Delta_{\Omega}}\right\|_{L^\infty(\Omega)}
\leq 
C_{10}
t^{-\frac{1}{2}}
\big(1+\log (1+t)\big)^{-1}
\|f\|_{L^2(\Omega)}, \quad t>0.
\end{align*}
\item[\bf (iii)] for every $f\in L^2(\Omega)\cap L_{d\mu}^1$,  
\begin{align*}\label{eq:basic-1}
\|e^{t\Delta_{\Omega}}f\|_{L^2(\Omega)}
\leq 
C_{10}'
t^{-\frac{1}{2}}
\big(1+\log (1+t)\big)^{-1}
\|f\|_{L_{d\mu}^1},\quad t>0, 
\end{align*}
where the constant $C_{10}'$ can be chosen as the same constant $C_{10}$ in {\bf (ii)}. 
\end{itemize}
\end{lemma}
	
\begin{remark}
The estimates {\bf (ii)} and {\bf (iii)}
are due to Grigor'yan and Saloff-Coste \cite{GS2002}. 
\end{remark}

We also need a modification of the Nash inequality 
with the measure $d\mu=H_\Omega\,dx$.
\begin{lemma}[{\cite{ISTWpre}}]
\label{lem:logNash}
There exists 
a positive constant $C_{11}$ such that 
for every $f\in H^1_0(\Omega)\cap L_{d\mu}^1$, 
\[
\|H_\Omega^{\frac{1}{2}}f\|_{L^2(\Omega)}^2
\leq 
C_{11}
\|f\|_{L_{d\mu}^1}
\|\nabla f\|_{L^2(\Omega)}.
\]
\end{lemma}

Decay estimates 
of the $L^2$-norm and the energy for 
this case and their optimality 
are the following. 
Here we use the notation 
\[
\mathbf{L}_{d\mu}^1=L_{d\mu}^1\times L_{d\mu}^1.
\] 
\begin{proposition}\label{prop:optimality5}
Let $\Omega$ be a two-dimensional exterior domain 
and let $u$ be the weak solution of \eqref{eq:Nd-ext-dw} with $(u_0,u_1)\in \mathbf{E}(\Omega)\cap \mathbf{L}_{d\mu}^1$. 
Then the following assertions hold: 

\begin{itemize}
\item[\bf (i)]
For every $m\in \Z_{\geq 0}$, 
there exists a positive constant $C_{12,m}$ depending only on 
$\Omega$ and $m$ such that 
\begin{align*}
\|u(t)-V_m(t)\|_{L^2(\Omega)}^2+tE_{\Omega}(u-V_m;t)
&\leq C_{12,m}t^{-1-2m}(\log t)^{-2}
\|(u_0,u_1)\|_{\mathbf{E}(\Omega)\cap \mathbf{L}_{d\mu}^1}^2,
\quad t\geq 2.
\end{align*}
where the family $\{V_m(\cdot)\}_{m\in\Z_{\geq 0}}$ is given 
in Definition \ref{def:asymptotics} with $A=-\Delta_{\Omega}$.
\item[\bf (ii)] If 
$\int_{\Omega}(u_0+u_1)h_\Omega\,dx\neq 0$,
then 
the following quantities are both finite and strictly positive:
\[
\limsup_{t\to\infty} \Big(t(\log t)^2\|u(t)\|_{L^2(\Omega)}^2\Big), 
\quad 
\limsup_{t\to\infty} \Big(t^{2}(\log t)^2E_{\Omega}(u;t)\Big).
\]
\item[\bf (iii)]If $u_1=-u_0$ and $\int_{\Omega}u_0h_\Omega\,dx\neq 0$, 
then 
the following quantities are both finite and strictly positive:
\[
\limsup_{t\to\infty} \Big(t^{3}(\log t)^2\|u(t)\|_{L^2(\Omega)}^2\Big), 
\quad 
\limsup_{t\to\infty} \Big(t^{4}(\log t)^2E_{\Omega}(u;t)\Big).
\]
\end{itemize}
\end{proposition}
\begin{proof}
{\bf (i)}
Applying Theorem \ref{intro:thm:abst-expansion}
with $n=m+1$, we find the following inequality:
\[
\|u(t)-V_{m+1}(t)\|_{L^2(\Omega)}^2+tE_{\Omega}(u-V_{m+1};t)
\leq C_{1,m+1}t^{-2m-2}
\|(u_0,u_1)\|_{\mathbf{E}(\Omega)}^2,
\quad t\geq 1.
\]
Noting that $u(t)-V_{m+1}(t)=u(t)-V_{m}(t)-v_{m}(t)$, 
we check the estimate for $v_{m}(t)$.
By Lemma \ref{lem:2dexteriorheat}, we have for $t\geq 2$, 
\begin{align*}
\|v_{m}(t)\|_{L^2(\Omega)}^2+tE_{\Omega}(v_{m};t)
&\leq Ct^{-2m}\Big(
\|e^{\frac{t}{2}\Delta_{\Omega}}v_{00}\|_{L^2(\Omega)}^2+\|e^{\frac{t}{2}\Delta_{\Omega}}u_{0}\|_{L^2(\Omega)}^2\Big)
\\
&\leq C_mt^{-1-2m}(\log t)^{-2}\|(u_0,u_1)\|_{\mathbf{L}_{d\mu}^1}^2
\end{align*}
for some constant $C_m$. 

{\bf (ii)} In this case, we adopt 
the strategy of the test function method 
in Ikeda--Sobajima \cite{IkSo-exterior-blowup}
to find a lower bound for $v_0(t)$. 
Let $\eta$, $\widetilde{\zeta}_\rho$ and $\widetilde{\xi}_\rho$ are as in the proof of Proposition \ref{prop:optimality3}
and let $h_\Omega$ be given in Lemma \ref{lem:harmonics2d}. 
Then we see that 
\begin{align*}
|\pa_t(\widetilde{\zeta}_\rho h_\Omega)+\Delta (\widetilde{\zeta}_\rho h_\Omega)|
&\leq 
\big(|\pa_t\widetilde{\zeta}_ \rho|+|\Delta \widetilde{\zeta}_\rho|\big)h_\Omega+2|\nabla \widetilde{\zeta}_\rho\cdot\nabla h_\Omega|
\leq
\frac{K}{\rho}\chi(\widetilde{\xi}_\rho)(1+h_\Omega)
\end{align*}
for some positive constant $K$; note that the difference compared to the case of Proposition \ref{prop:optimality1} $(N\geq 3)$ is just $h_\Omega\notin L^\infty(\Omega)$. 
Choose $\rho_*\geq 2$ such that 
for every $\rho>\rho_*$, 
\[
\left|
\int_{\Omega}v_0\widetilde{\zeta}_\rho h_\Omega\,dx
\right|
\geq c_0
=
\frac{1}{2}\left|
\int_{\Omega}v_0h_\Omega\,dx
\right|. 
\]
Then a computation similar to the proof of Proposition \ref{prop:optimality1} gives
for every $\rho>\rho_*$
\begin{align*}
c_0^2
	&\leq 
\left(\iint_{\Omega\times \R_+}|\pa_t \widetilde{\zeta}_\rho+\Delta \widetilde{\zeta}_\rho|^{2}\,dxdt\right) \rho Y'(\rho)\leq  
K\rho\log \rho \widetilde{Y}'(\rho), 
\end{align*}
where we have put 
\[
\widetilde{Y}(T)=\int_0^T\left(\iint_{\Omega\times \R_+}|V_0|^2\mathbbm{1}_{[1/2,1]}(\widetilde{\xi}_\rho)(1+h_\Omega)\,dxdt\right)\frac{d\rho}{\rho}.
\]
This implies that 
\[
\widetilde{Y}(\rho_*)+\frac{c_0^2}{K}\log\left(\frac{\log T}{\log \rho_*}\right)\leq \widetilde{Y}(T)
\leq 
\log 2 \int_0^T\|H_{\Omega}^{\frac{1}{2}}V_0(t)\|_{L^2(\Omega)}^2\,dt.
\]
Therefore Lemmas \ref{lem:logNash} and  \ref{lem:2dexteriorheat} give 
\[
\frac{c_0^2}{KC_{8} \log 2 \|v_0\|_{L_{d\mu}^1}}\log\left(\frac{\log T}{\log \rho_*}\right)
\leq 
\int_0^T\|\nabla V_0(t)\|_{L^2(\Omega)}\,dt.
\]
This implies $\limsup_{t\to\infty}\big(t^2(\log t)^2E_{\Omega}(v;t)\big)>0$. 
The estimate for $L^2$-norm can be found via 
$\|\nabla v_0(t)\|_{L^2(\Omega)}^2\leq 
2t^{-1}\|v_0(t/2)\|_{L^2(\Omega)}^2$.

{\bf (iii)} 
In this case we note again that $v_0(t)\equiv 0$ and therefore $V_1(t)=v_1(t)=\Delta e^{t\Delta_{\Omega}}u_0$. 
As in {\bf (ii)}, 
in this case we have
\[
\limsup_{t\to\infty}\Big(t^2(\log t)^2\|\nabla e^{t\Delta_\Omega}u_0\|_{L^2(\Omega)}^2\Big)>0.
\]
Then it is enough to observe that 
\begin{align*}
t^2(\log t)^2\|\nabla e^{t\Delta_\Omega}u_0\|_{L^2(\Omega)}^2
&\leq 
\big(t(\log t)^2\|e^{t\Delta_\Omega}u_0\|_{L^2(\Omega)}^2\big)^{\frac{1}{2}}
\big(t^{3}(\log t)^2\|\Delta e^{t\Delta_\Omega}u_0\|_{L^2(\Omega)}^2\big)^{\frac{1}{2}}
\\
&\leq 
C_7\|u_0\|_{L_{d\mu}^1}\big(t^{3}(\log t)^2\|v_1(t)\|_{L^2(\Omega)}^2\big)^{\frac{1}{2}}
\end{align*}
and 
\begin{align*}
t^{3}(\log t)^2\|v_1(t)\|_{L^2(\Omega)}^2
&\leq 
\big(t^{\frac{1}{2}}(\log t)\|\nabla e^{t\Delta_\Omega}u_0\|_{L^2(\Omega)}\big)^{\frac{2}{3}}
\big(t^{4}(\log t)^2\|\nabla \Delta e^{t\Delta_\Omega}u_0\|_{L^2(\Omega)}^2\big)^{\frac{2}{3}}
\\
&\leq 
(C_7\|u_0\|_{L_{d\mu}^1})^{\frac{2}{3}}
\big(t^{4}(\log t)^2E_{\Omega}(v_1;t)\big)^{\frac{2}{3}}.
\end{align*}
The proof is complete.
\end{proof}

\subsection{Estimates for the local energy}

For $R>R_\Omega$, we consider the large-time behavior 
of a localized version of the energy functional given by 
\[
E_{\Omega,R}(u;t)=\int_{\Omega\cap B(0,R)}\Big(|\pa_tu(x,t)|^2+|\nabla u(x,t)|^2\Big)\,dx
\]
for $R>R_\Omega$ as an application of asymptotic expansion (Theorem \ref{intro:thm:abst-expansion}). 
Here we again use the space $\mathbf{L}^q(\Omega)=L^q(\Omega)\times L^q(\Omega)$ 
and do not restrict ourselves to the case of compactly supported functions.
The first assertion is for the large-time behavior of 
the local energy in $\Omega=\R^N$.

\begin{proposition}
Let $\Omega=\R^N$ $(N\in \N)$.
For every $q\in [1,2)$ and $m\in \Z_{\geq 0}$, 
there exists a positive constant $C_{10,q,m}$ such that 
the following assertion holds:
If $(u_0,u_1)\in \mathbf{E}(\R^N)\cap \mathbf{L}^q(\R^N)$,  
then the corresponding weak solution of \eqref{eq:Nd-ext-dw} 
satisfies for  every $R>0$, 
\begin{align*}
E_{\R^N,R}(u-V_m;t)\leq C_{10,q,m}R^N
t^{-\frac{N}{q}-2m-1}
\|(u_0,u_1)\|_{\mathbf{E}(\R^N)\cap \mathbf{L}^q(\R^N)}^2, \quad t\geq 1,
\end{align*}
where the family $\{V_m(\cdot)\}_{m\in\Z_{\geq 0}}$ 
in Definition \ref{def:asymptotics} with $A=-\Delta=-\Delta_{\R^N}$.
\end{proposition}

\begin{proof}
Fix $R>0$ arbitrary.
We apply Theorem \ref{intro:thm:abst-expansion} with $n=m+N$, that is, 
$u-V_{m+N}$ satisfies 
\[
E_{\R^N,R}(u-V_{m+N};t)
\leq 
E_{\R^N}(u-V_{m+N};t)\leq C_{4,m+N}t^{-2N-2m-1}\|(u_0,u_1)\|_{\mathbf{E}(\R^N)}^2, \quad t\geq 1.
\]
Therefore it suffices to consider the estimates for $v_\ell$ ($\ell=m,\ldots, m+N-1$).
By the basic property of heat semigroup, 
we see for $t\geq 1$, 
\begin{align*}
E_{\R^N,R}(v_\ell;t)
&\leq 
\Big(\|\pa_tv_\ell(t)\|_{L^\infty(\R^N)}^2+\|\nabla v_\ell(t)\|_{L^\infty(\R^N)}^2\Big)
\|\mathbbm{1}_{B(0,R)}\|_{L^1(\R^N)}
\\
&\leq 
CR^{N}t^{-\frac{N}{q}-2\ell-1}\|(u_0,u_1)\|_{\mathbf{L}^q(\R^N)}^2.
\end{align*}
Consequently, the estimate for $E_{\R^N,R}(u-V_{m};t)$ 
only reflects the worst one $E_{\R^N,R}(v_{m};t)$. 
The proof is complete.
\end{proof}

For the exterior problem, 
the gradient of the heat kernel 
does not have a good decay property in general (see e.g., 
Gimbo--Sakaguchi \cite{GiSa1994} and also Georgiev--Taniguchi \cite{GiTa2019}).
The following assertion 
can be regarded as a modification of the result in \cite{Shibata1983,DaSh1995}.
A precise dependence of the parameter $R$ for the constant 
is explicitly given.

\begin{proposition}\label{prop:Nd-exterior}
Let $\Omega$ be an exterior domain $(N\geq 2)$. 
For every $q\in [1,2)$ and $m\in \Z_{\geq 0}$, 
there exists a positive constant $C_{13,q,m}$ such that 
the following assertion holds:
If $(u_0,u_1)\in \mathbf{E}(\Omega)\cap \mathbf{L}^q(\Omega)$, 
then the corresponding weak solution $u$ of \eqref{eq:Nd-ext-dw}
satisfies for every $R>R_\Omega$, 
\[
E_{\Omega,R}(u-V_{m};t)\leq C_{13,q,m}R^{N-2}
t^{-\frac{N}{q}-2m}
\|(u_0,u_1)\|_{\mathbf{E}(\Omega)\cap \mathbf{L}^q(\Omega)}^2, \quad t\geq 1.
\]
where the family $\{V_m(\cdot)\}_{m\in\Z_{\geq 0}}$ is given 
in Definition \ref{def:asymptotics} with $A=-\Delta_{\Omega}$.
\end{proposition}

\begin{proof}
Fix $R>R_\Omega$ arbitrary. 
We apply Theorem \ref{intro:thm:abst-expansion} with $n=m+N-2$, that is, 
$u-V_{m+N-2}$ satisfies 
\[
E_{\Omega,R}(u-V_{m+N-2};t)\leq E_{\Omega}(u-V_{m+N-2};t)\leq C_{1,m+N-2}t^{-2N-2m+3}\|(u_0,u_1)\|_{\mathbf{E}(\Omega)}^2, \quad t\geq 1.
\]
Therefore it again suffices to consider the estimates for $v_\ell$ ($\ell=m,\ldots, m+N-3$).
Let $\eta$ be as in the proof of Proposition \ref{prop:optimality1}
and put 
\[
\zeta_R(x)=\eta\left(\frac{|x|^2}{2R^2}\right)
\]
which satisfies 
$\mathbbm{1}_{B(0,R)}\leq \zeta_R \leq \mathbbm{1}_{B(0,2R)}$. 
Then by integration by parts, we can compute
\begin{align*}
\int_{\Omega\cap B(0,R)}|\nabla v_\ell|^2\,dx
&\leq 
\int_{\Omega}|\nabla v_\ell|^2\zeta_R\,dx
\\
&= 
\frac{1}{2}\int_{\Omega}v_\ell^2\Delta \zeta_R\,dx-\int_{\Omega}v_\ell \Delta v_\ell \zeta_R\,dx
\\
&\leq \frac{1}{2}\|v_\ell\|_{L^\infty(\Omega)}^2\|\Delta \zeta_R\|_{L^1(\Omega)}
+\|v_\ell\|_{L^N(\Omega)}\|\Delta v_\ell\|_{L^N(\Omega)}\|\zeta_R\|_{L^{1^{**}}(\Omega)}
\\
&\leq CR^{N-2}t^{-\frac{N}{q}-2\ell}\|(u_0,u_1)\|_{\mathbf{L}^q(\Omega)}^2
\end{align*}
where we have put $1^{**}=\frac{N}{N-2}$ $(N\geq 3)$ and $1_{**}=\infty$ $(N=2)$. 
And also 
\begin{align*}
\int_{\Omega\cap B(0,R)}|\pa_tv_\ell|^2\,dx
&\leq \|\pa_tv_\ell\|_{L^{N}(\Omega)}^2
\|\mathbbm{1}_{B(0,R)}\|_{L^{1^{**}}(\Omega)}.
\\
&\leq CR^{N-2}t^{-N(\frac{1}{q}-\frac{1}{N})-2\ell-2}\|(u_0,u_1)\|_{\mathbf{L}^{q}(\Omega)}^2.
\\
&=CR^{N-2}t^{-\frac{N}{q}-2\ell-1}\|(u_0,u_1)\|_{\mathbf{L}^q(\Omega)}^2.
\end{align*}
As a consequence, it is clarified that the factor 
$E_{\Omega,R}(v_m;t)$ has the worst decay rate $t^{-\frac{N}{q}-2m}$. 
The proof is complete.
\end{proof}

To close the paper, 
we finally discuss some refinement 
of the decay rate of $E_{\Omega,R}(u,t)$ 
when $\Omega$ is a two-dimensional exterior domain.
Thanks to the decay estimate for 
the Dirichlet heat semigroup 
involving the logarithmic function with 
initial values in $L_{d\mu}^1$ (Lemma \ref{lem:2dexteriorheat}), 
we also obtain a decay estimate involving the logarithmically decaying factor $(\log t)^{-4}$.

\begin{proposition}
Let $\Omega$ be a two-dimensional exterior domain 
and let $u$ be the solution of \eqref{eq:Nd-ext-dw} with 
$(u_0,u_1)\in \mathbf{E}(\Omega)\cap \mathbf{L}_{d\mu}^1$.  
Then for every $m\in\Z_{\geq 0}$, 
there exists a positive constant $C_{14,m}$ (depending only on $\Omega$ and $m$) such that 
for every $R>R_\Omega$, 
\[
E_{\Omega,R}(u-V_{m};t)\leq 
C_{14,m}\left(1+\log \frac{R}{r_\Omega}\right)^2
t^{-2-2m}(\log t)^{-4}
\|(u_0,u_1)\|_{\mathbf{E}(\Omega)\cap \mathbf{L}_{d\mu}^1}^2, 
\quad t\geq 2, 
\]
where the family $\{V_m\}_{m\in\Z_{\geq 0}}$ is given 
in Definition \ref{def:asymptotics} with $A=-\Delta_{\Omega}$ and $r_{\Omega}=\inf\limits_{x\in\pa\Omega}|x|$.
\end{proposition}
\begin{proof}
Fix $R>R_\Omega$ arbitrary.
We apply Theorem \ref{prop:Nd-exterior} with $m$ replaced with $m+1$. 
Then we already have
the following estimate for $u-V_{m+1}=u-V_m-v_m$:
\[
E_{\Omega,R}(u-V_{m+1};t)\leq C_{13,1,m+1}
t^{-2m-4}
\|(u_0,u_1)\|_{\mathbf{E}(\Omega)\cap \mathbf{L}^1(\Omega)}^2, \quad t\geq 1.
\]
Then we consider the estimate for $E_{\Omega,R}(v_m;t)$
via Lemma \ref{lem:2dexteriorheat}.
Using the same function $\zeta_R$ as in Proposition \ref{prop:Nd-exterior}, 
we see that 
\begin{align*}
&\int_{\Omega\cap B(0,R)}|\nabla v_m(t)|^2\,dx
\\
&\leq  
\frac{1}{2}\int_{\Omega}v_m(t)^2\Delta \zeta_R\,dx-\int_{\Omega}v_m(t) \Delta v_m(t) \zeta_R\,dx
\\
&\leq 
\frac{1}{2}\|H_\Omega^{-1}v_m(t)\|_{L^\infty(\Omega)}^2\|H_\Omega^2\Delta \zeta_R\|_{L^1(\Omega)}
+
\|H_\Omega^{-1}v_m(t)\|_{L^2(\Omega)}
\|H_\Omega^{-1}\Delta v_m(t)\|_{L^2(\Omega)}
\|H_\Omega^2\zeta_R\|_{L^\infty(\Omega)}
\\
&\leq 
C\left(1+\log \frac{R}{r_\Omega}\right)^2
\Big(
\|H_\Omega^{-1}v_m(t)\|_{L^\infty(\Omega)}^2
+
\|H_\Omega^{-1}v_m(t)\|_{L^2(\Omega)}
\|H_\Omega^{-1}\Delta v_m(t)\|_{L^2(\Omega)}
\Big).
\end{align*}
Here we put $v_{m*}(t)$ as 
\[
v_{m*}(t)=
\sum_{j=0}^m
\begin{pmatrix}
2m \\ m+j
\end{pmatrix}\frac{(t\Delta)^j}{j!}e^{\frac{2t}{3}\Delta_\Omega}(u_0+u_1)
-
\sum_{k=0}^{m-1}
\begin{pmatrix}
2m-1\\ m+k
\end{pmatrix}\frac{(t\Delta)^k}{k!}e^{\frac{2t}{3}\Delta_\Omega}u_0
\]
which satisfies $v_m(t)=e^{\frac{t}{3}\Delta_\Omega}(-\Delta)^{m}v_{m*}(t)$.
Putting $s=t/3$, we can control the first term 
on the right hand side of the above inequality as
\begin{align*}
\|H_\Omega^{-1}v_m(t)\|_{L^\infty(\Omega)}^2
&\leq 
C_{10}s^{-1}(\log s)^{-2}\|(-\Delta)^m v_{m*}(t)\|_{L^2(\Omega)}^2
\\
&\leq 
CC_{10}s^{-1}(\log s)^{-2-2m}
\Big(
\|e^{s\Delta_\Omega}(u_0+u_1)\|_{L^2(\Omega)}^2
+\|e^{s\Delta_\Omega}u_0\|_{L^2(\Omega)}^2\Big)
\\
&\leq C'C_{10}^2s^{-2-2m}(\log s)^{-4}\|(u_0,u_1)\|_{\mathbf{L}_{d\mu}^1}^2.
\end{align*}
On the one hand, 
we see that for $t>2R_\Omega^2$ and $f\in L^2(\Omega)\cap L_{d\mu}^1$, 
\begin{align*}
\|H_\Omega^{-1}e^{s\Delta_{\Omega}}f\|_{L^2(\Omega)}
&\leq 
\|H_\Omega^{-1}e^{s\Delta_{\Omega}}f\|_{L^2(\Omega \cap B(0,\sqrt{s}))}
+
\|H_\Omega^{-1}e^{s\Delta_{\Omega}}f\|_{L^2(B(0,\sqrt{s})^c)}
\\
&
\leq \|H_\Omega^{-1}e^{s\Delta_{\Omega}}f\|_{L^\infty(\Omega)}
\|\mathbbm{1}\|_{L^2(\Omega \cap B(0,\sqrt{s}))}
+
\|H_\Omega^{-1}\|_{L^\infty(B(0,\sqrt{s})^{c})}\|e^{s\Delta_{\Omega}}f\|_{L^2(\Omega)}
\\
&
\leq CC_{10}(\log s)^{-1}\|f\|_{L^2(\Omega)}.
\end{align*}
we see that the remaining term can be estimated as follows:for $\tau=0,1$, 
\begin{align*}
\|H_\Omega^{-1}(-\Delta)^\tau v_m(t)\|_{L^2(\Omega)}
&\leq 
CC_{10}(\log s)^{-1}\|(-\Delta)^{m+\tau}v_{m*}(t)\|_{L^2(\Omega)}
\\
&\leq 
C'C_{10}s^{-m-\tau}(\log s)^{-1}
\Big(
\|e^{\frac{t}{3}\Delta_\Omega}(u_0+u_1)\|_{L^2(\Omega)}
+\|e^{\frac{t}{3}\Delta_\Omega}u_0\|_{L^2(\Omega)}\Big)
\\
&\leq 
C''C_{10}^2s^{-m-\tau-\frac{1}{2}}(\log s)^{-2}\|(u_0,u_1)\|_{\mathbf{L}_{d\mu}^1}^2.
\end{align*}
	The proof is complete.
\end{proof}

\subsection*{Declarations}
\begin{description}
\item[Data availability] Data sharing not applicable to this article as no datasets were generated or analysed during the current study.
\item[Conflict of interest] The author declares that he has no conflict of interest.
\end{description}


\newpage 
\end{document}